\documentclass[aop,MSNbibl,dvips]{arximspdf}
\usepackage{graphicx}
\doi{10.1214/12-AOP770} 
\volume{41}
\issue{4}
\pubyear{2013}
\firstpage{2682}
\lastpage{2708}

\makeatletter
\setattribute{abstract}{width}{280pt}

\newproclaim{remark}{Remark}
\newproclaim{remarks}{Remarks}

\newtheorem{theorem}{Theorem}[section]
\newtheorem{corollary}[theorem]{Corollary}
\newtheorem{prop}[theorem]{Proposition}
\newproclaim{definition}[theorem]{Definition}

\newcommand{\rrvert}{\vert}
\newcommand{\llvert}{\vert}

\newcommand{\cal}{\mathcal}

\newcommand{\Sp}{Z}
\newcommand{\Sm}{A}

\newcommand{\cip}{\stackrel{p}{\longrightarrow}}
\newcommand{\ed}{\stackrel{d}{=}}

\newcommand{\aaa}{\alpha}
\renewcommand{\th}{\theta}
\newcommand{\al}{\alpha}
\newcommand{\Om}{\Omega}

\newcommand{\reals}{\mathbb{R}}

\newcommand{\lar}{\Longrightarrow}

\newcommand{\eps}{\varepsilon}
\newcommand{\si}{\sigma}

\newcommand{\N}{\mathbb{N}}
\newcommand{\R}{\mathbb{R}}
\newcommand{\Z}{\mathbb{Z}}
\newcommand{\F}{\mathcal{F}}

\renewcommand{\P}{\mathbf{P}}
\newcommand{\E}{\mathbf{E}}
\newcommand{\EEE}{\mathcal{E}}
\renewcommand{\AA}{{\cal A}}
\newcommand{\BB}{{\cal B}}
\newcommand{\CC}{{\cal C}}
\newcommand{\DD}{{\cal D}}
\newcommand{\FF}{{\cal F}}
\renewcommand{\SS}{{\cal S}}

\makeatother

\begin{document}
\begin{frontmatter}

\title{Random walks at random times: Convergence to iterated
L\'{e}vy motion, fractional stable motions, and other self-similar
processes}
\runtitle{Random walks at random times}

\begin{aug}
\author[A]{\fnms{Paul} \snm{Jung}\corref{}\thanksref{t1}\ead[label=e1]{pauljung@gmail.com}}
\and
\author[B]{\fnms{Greg} \snm{Markowsky}\thanksref{t2}}
\runauthor{P. Jung and G. Markowsky}
\affiliation{University of Alabama at Birmingham and Monash University}
\address[A]{Department of Mathematics\\
University of Alabama at Birmingham\\
CH 452\\
1720 2nd Ave. South\\
Birmingham, Alabama 35294\\
USA} 
\address[B]{Department of Mathematical Sciences\\
Monash University\\
Clayton, Victoria 3800\\
Australia}
\end{aug}

\thankstext{t1}{Research started while at Sogang University and
supported by Sogang University research Grant 200910039.}

\thankstext{t2}{Supported by the Priority Research
Centers Program through the National Research Foundation of Korea
(NRF) funded by the Ministry of Education, Science and Technology
(Grant \#2009-0094070) and from Australian Research Council Grant
DP0988483.}

\received{\smonth{5} \syear{2011}}
\revised{\smonth{4} \syear{2012}}

%
\begin{abstract}
For a random walk defined for a doubly infinite sequence of times, we
let the time parameter itself be an integer-valued process, and call
the orginal process a random walk at random time. We find the scaling
limit which generalizes the so-called iterated Brownian motion.

Khoshnevisan and Lewis [\textit{Ann. Appl. Probab.} \textbf{9} (1999)
629--667] suggested ``the existence of a form of measure-theoretic
duality'' between iterated Brownian motion and a Brownian motion in
random scenery. We show that a random walk at random time can be
considered a random walk in ``alternating'' scenery, thus hinting at a
mechanism behind this duality.

Following Cohen and Samorodnitsky [\textit{Ann. Appl. Probab.}
\textbf{16} (2006) 1432--1461], we also consider alternating random
reward schema associated to random walks at random times. Whereas
random reward schema scale to local time fractional stable motions, we
show that the alternating random reward schema scale to indicator
fractional stable motions.

Finally, we show that one may recursively ``subordinate'' random time
processes to get new local time and indicator fractional stable motions
and new stable processes in random scenery or at random times. When
$\alpha=2$, the fractional stable motions given by the recursion are
fractional Brownian motions with dyadic $H\in(0,1)$. Also, we see that
``un-subordinating'' via a time-change allows one to, in some sense,
extract Brownian motion from fractional Brownian motions with $H<1/2$.
\end{abstract}

%

\end{frontmatter}

\section{Introduction}
Let $B^{(i)}(t), i=1,2,3$, be three independent Brownian motions, and
let a two-sided Brownian motion be defined by
%
\begin{equation}
\label{two-sided} \tilde{B}(t):= \cases{ B^{(1)}(t), &\quad if $t\ge0$,
\cr
B^{(2)}(-t), &\quad if $t<0$.}
\end{equation}
In~\cite{burdzy1992some}, Burdzy studied the process $ (\tilde
{B}(B^{(3)}(t)) )_{t\ge0}$ which he called an iterated Brownian
motion (IBM). It can be thought of as a two-sided Brownian motion which
is nonmonotonically ``subordinated'' to another
Brownian motion.
This process was also used by Deheuvels and Mason
\cite{deheuvels1992functional} to study the Bahadur--Kiefer process.
Also, a variant of IBM, where the pure imaginary process
$i B^{(2)}(-t)$ was substituted for $t<0$, was utilized by Funaki \cite
{funaki1979probabilistic} to study the PDE
%
\begin{equation}
\frac{\partial u}{\partial t}=\frac{1}{8}\,\frac{\partial^4
u}{\partial x^4}.
\end{equation}

Recently, more general \textit{processes at random times} called
$\alpha$-time Brownian motions and $\alpha$-time fractional Brownian
motions were introduced in~\cite{nane2006laws,nane2011local}. In
these works (along with several references therein), the connection
between processes at random times and various PDEs was studied,
along with the local time and path properties of the iterated
processes. In a different direction, the scaling and asymptotic
density of a discretized version of IBM called iterated random walk
was analyzed in the physics literature~\cite{turban2004iterated}.

In this work, we consider generalizations of the iterated random walk
which we call \textit{random walks at random times} (RWRT)
and \textit{dependent walks at random times} (DWRT) and relate them with a
different portion of the probability literature
concerning random walks in random scenery. This relation was first
noted by Khoshnevisan and Lewis~\cite{khoshnevisan1996iterated} who
stated that
there was ``a surprising connection between
the variations (of IBM) and H. Kesten and F. Spitzer's Brownian motion
in random
scenery.'' Later, in~\cite{khoshnevisan1999stochastic}, a form of
measure-theoretic duality was shown between the two processes.
Here, we present a mechanism on the discrete level which shows a
connection between the two processes.

We show that under suitable conditions, the scaling limits of RWRT
and DWRT are ($H$-sssi)-time $\alpha$-stable L\'{e}vy motions, a new
class of processes at random times. If $X(t)$ is a two-sided
$\alpha$-stable L\'{e}vy motion defined similarly to (\ref{two-sided}),
and $Y_t$ is an independent $\alpha$-stable L\'{e}vy motion, then we
call $X(Y_t)$ an \textit{iterated L\'{e}vy motion}. If, more generally,
$Y_t$ is an independent $H$-self-similar, stationary-increment
process (sssi), then an ($H$-sssi)-time $\alpha$-stable L\'{e}vy motion
is given by $X(Y_t)$. Assuming $0<H<1$, we will see that $X(Y_t)$ is
an $H/\alpha$-sssi process with Hurst exponent less than $1/\alpha$. They
naturally complement stable processes in random scenery which are
the limiting continuous processes of~\cite{KS} and~\cite{Wang03} and
which have Hurst exponents greater than $1/\alpha$ (Wang~\cite{Wang03}
considered only the case $\alpha=2$, but this was extended to $\alpha<2$
by Cohen and Dombry~\cite{CD}).

Random walks in random scenery (RWRS) and their scaling limits,
stable processes in random scenery, were first introduced
independently in~\cite{KS,Borodin}. The purpose of~\cite{KS} was to
introduce a new class of sssi processes given by the scaling limits
of RWRS. The scaling limits have integral representations as stable
integrals of local time kernels (of a process $Y_t$). When the
random scenery are $\alpha$-stable laws, they scale to the $\alpha$-stable
random measure against which the local time kernel is integrated. In
comparison, there is also an integral representation of
($H$-sssi)-time $\alpha$-stable L\'{e}vy motions given by the stable
integration of random kernels of type $1_{[0,Y_t]}$ against
$\alpha$-stable random measures.

When $Y_t$ is a generic $H$-sssi process, the stable processes in
random scenery discussed
above also include the model of~\cite{Wang03}.
Wang used ``dependent walks'' to collect the scenery, instead of random
walks, leading to a dependent walk in
random scenery (DWRS).
In particular, the dependent walks he used were discrete-time Gaussian
processes known to scale
to fractional Brownian motion (fBm).

Random reward schema are sums of independent copies of discrete
processes in random scenery.
In~\cite{CS,DG,CD} it was shown that the random reward schema of RWRS
and DWRS scale
to $H$-sssi symmetric $\alpha$-stable ({S$\al$S}) processes called
\textit{local time fractional {S$\al$S} motions} (with $H>1/\alpha$).
In this work, we show that the scaling limits of random reward schema
for RWRT and DWRT are $H$-sssi {S$\al$S} processes called
\textit{indicator fractional {S$\al$S} motions} (with $H<1/\alpha$) which
were introduced in~\cite{jung2010indicator}.

Note that fBm is the only sssi Gaussian process. Thus, when the scenery
has finite variance and $\alpha=2$,
local time fractional {S$\al$S} motions and indicator fractional
{S$\al$S} motions
reduce to fBm with $H>1/2$ and $H<1/2$, respectively.

As will be seen in Section~\ref{secmodels}, the mechanism behind the
connection between local time
and indicator fractional stable motions is the same as the mechanism
which connects Brownian motion in random
scenery (BMRS) with IBM.
In effect, the mechanism shows that the indicator kernels of the latter
processes
can be thought of as ``alternating'' versions of the local time kernels
of the former.

Together, local time fractional {S$\al$S} motions and indicator fractional
stable motions form a class of
fractional stable motions ($H$-sssi {S$\al$S} processes) which may be
thought of as one of several generalizations
of fractional Brownian motion. Their increment processes are stationary
and have the ergodic-theoretic property
of being null conservative, a concept introduced in \cite
{samorodnitsky2005null}.
This property distinguishes them from fractional stable motions which
have dissipative or positive conservative increment processes.
The most well-known examples of fractional stable motions with
dissipative or positive conservative increment processes
are the linear fractional stable motions and the real harmonizable
stable motions, respectively, as can be seen in Figure~\ref{figur1}.

\begin{figure}

\includegraphics{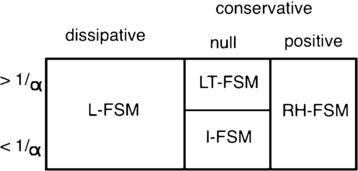}

\caption{$\alpha\in(1,2)$:
LT${} = {}$local time, I${} = {}$indicator,
L${} = {}$linear, RH${} = {}$real harmonizable.}\label{figur1}
\end{figure}

We also consider single-scenery random reward schema introduced in
\cite{dombry14functional}. Here we again take sums of identically
distributed RWRTs or DWRTs. However, the copies have a dependence
structure since they use the same ``single scenery.'' This
dependence will be made more explicit below. The scaling limits of
single-scenery random reward schema of RWRS and DWRS no longer have
stationary increments; however, they are easily seen to be $H$-ss
{S$\al$S} processes with $H>1/\alpha$. Similarly, the scaling limits of
single-scenery random reward schema of RWRT and DWRT are $H$-ss
{S$\al$S} processes with $H<1/\alpha$.

Finally, we also present a recursive construction of some local time
and indicator fractional stable motions. In particular, we show
that at each step of the recursion, the local times exist and are in
$L^2( \Om\times\R)$. The recursively defined processes give the first
examples of local time fractional stable motions for which the
processes collecting the scenery are neither fBm nor $\beta$-stable
L\'{e}vy motions. In the case $\alpha=2$, the processes are given by
integrals against Gaussian random measures, and the recursion
constructs fBm, of any dyadic Hurst parameter, using one Brownian
motion and a countable family of independent random Gaussian
measures.

As mentioned above, RWRT and, in particular, its scaling limit are in
some sense nonmonotonically subordinated processes. Usually one may
not undo a subordination---for example, one
can embed a stable process in Brownian motion, but cannot extract
Brownian motion from the stable process since the filtration is
strictly smaller.
However, we will see that when the scaling limit of the random time
process, $Y_t$, is fBm, one can undo the subordination using the
time-change $\tau_s = \inf_{t \geq0} \{t\dvtx Y_t=s\}$.
Extending such a time-change procedure to the kernels of indicator
fractional stable motions when $\alpha=2$, we
find that one can, in some sense, extract Brownian motion from
fractional Brownian motions satisfying $H<1/2$.

The rest of the paper is arranged as follows. In
Section~\ref{secmodels} we describe RWRTs and RWRSs. We also describe
their respective random reward schema and scaling limits. The section
ends with a statement describing new scaling limit results. The proofs
of the scaling weak convergence results are given in Section
\ref{secproofs}. In Section~\ref{secconstruction}, we describe the
recursive construction mentioned above, and complete the nontrivial
task of showing that the recursion produces processes that are well
defined. The main component of this task is showing that the local
times exist and are in $L^2(\Om\times\R)$. Finally, in Section~\ref{sectime-changes} we explain how to extract Brownian motion from
fBm with any Hurst parameter satisfying $H<1/2$.

\section{Discrete and continuous models}\label{secmodels}
\subsection{Random walks at random times and alternating random reward schema}
We start with a simple description of RWRS. Let
$\{\eta_\alpha(k)\}_{k\in\Z}$
be a set of i.i.d. symmetric random variables in the domain of
attraction of an {S$\al$S} law, $\alpha\in(0,2]$ with scale parameter
$\sigma=1$.
The family $\{\eta_\alpha(k)\}$
depicts the \textit{scenery} associated to the vertices of $\Z$.
Let
%
\begin{equation}
W(n):= \sum_{k=1}^n
\xi_\beta(k)
\end{equation}
be a symmetric random walk on $\Z$
with steps $\xi_\beta(k)$ in the domain of attraction of an {S$\beta
$S}
law, $\beta\in(1,2]$.
The random walk roams amidst the scenery $\{\eta_\alpha(k)\}$ which
are independent from the steps $\{\xi_\beta(k)\}$.

The cumulative scenery process
%
\begin{equation}
\label{deftotalreward} \Sp_n=\Sp_n(
\eta_\alpha,W):=\sum_{k=1}^n
\eta_\alpha\bigl(W(k)\bigr)
\end{equation}
is called a \textit{random walk in random scenery}. The scenery $\{\eta
_\alpha(k)\}$ can alternatively be
thought of as random \textit{reward} collected by the random walk when it
visits vertex~$k$.

We note that some authors call the pair $(W,\eta_\alpha(W))$ a RWRS
process (e.g.,~\cite{den2006random}).
Since most of the papers cited in this work refer to (\ref{deftotalreward}) as the RWRS, we stick with this notation.

Wang~\cite{Wang03} considered a slight modification of RWRS by using a
discrete approximation of a Gaussian process instead of a random walk:
%
\begin{equation}
\label{defvertexprocess} \Sp_n=\Sp_n(
\eta_\alpha,G_H):= \sum_{k=1}^n
\eta_\alpha\bigl(\bigl\lceil G_H(k)\bigr\rceil\bigr).
\end{equation}
Here $\lceil\cdot\rceil$ is the ceiling function, and $G_H(k)$ is
the partial sum of a stationary Gaussian process $X_k$ with
correlations $r(j-k) = \E X_jX_k$ satisfying
%
\begin{equation}
\label{condcorrelations} \sum_{j=1}^n
\sum_{k=1}^n r(j-k)\sim n^{2H},
\end{equation}
where $0<H<1$.
In addition to (\ref{defvertexprocess}), there have been myriad
generalizations of~(\ref{deftotalreward}),
and we refer the reader to the introduction of~\cite{guillotinlimit}
for a nice summary of such generalizations.

We refer to (\ref{defvertexprocess}) as a \textit{dependent walk in
random scenery} (DWRS).
In general, we consider $\Sp_n(\eta_\alpha,W_H)$ for which the
\textit{collecting process} $W_H(n)$ has stationary
increments and also satisfies the following {scaling limit properties}:
{\renewcommand{\theequation}{SLP}
\begin{equation}
\label{Wgamma}\quad
\cases{\mbox{\hphantom{ii}(i) } \displaystyle \lim_{n\to\infty}n^{-H}W_H
\bigl(\lfloor nt\rfloor\bigr)\Rightarrow{Y}_t, \qquad
\mbox{in $\DD\bigl([0,
\infty)\bigr)$},
\vspace*{2pt}\cr
\mbox{\hphantom{i}(ii) } {Y}_t \mbox{ is a nondegenerate }H
\mbox{-sssi process } \vspace*{2pt}\cr
\hspace*{164pt}\mbox{(${Y}_0=0$ by self-similarity)},
\vspace*{2pt}\cr
\mbox{(iii) }
\E|{Y}_t|<\infty,}
\end{equation}}

\noindent where $\DD([0,\infty))$ is equipped with the usual Skorohod topology
(also called the $J_1$-topology).

The condition that ${Y}_t$ be sssi guarantees that $\Sp_n$ scales to
an sssi process as well,
and this was in fact the original motivation of introducing $\Sp_n$ in~\cite{KS}.
Note that we use the stable parameter $\alpha\in(0,2]$ for the
scenery/reward and
consequently the increments of the RWRS/DWRS; however, we reserve the
stable parameter $\beta\in(1,2]$
for the increments of the collecting process (note that we require
$\beta>1$ in order to guarantee $\E|Y_t|<\infty$).

We introduce a variant of $\Sp_n$ in which the reward alternate in
sign and are associated with edges instead of vertices.
In our variant of RWRS, we use symmetric reward $\{\eta_\alpha(e)\}$
together with signs $\{\sigma_e\}, \sigma_e\in\{-1,+1\}$,
associated to the edge set of $\Z$. At time zero, all signs are plus
one, $\sigma_e(0)=+1$;
however, $(\sigma_e(n))_{n\ge0}$ is a process determined by the
collecting process in a manner discussed below.

Consider a discrete collecting process $W_H(n)$ satisfying condition
(\ref{Wgamma}). Note that our definition allows $|W_H(n)-W_H(n-1)|$ to
be greater than one. Let $\EEE_n$ be the set of connected edges
traversed on the $n$th step of $W_H(n)$, that is, the set of edges
between $W_H(n-1)$ and $W_H(n)$ [thus $\EEE_n$ has cardinality
$|W_H(n)-W_H(n-1)|$]. At the $n$th step, the process
$W_H(n)$:
\begin{itemize}
\item earns the signed reward $\sigma_e(n-1)\cdot\eta_\alpha(e)$
of all edges $e\in\EEE_n$ and then
\item reverses the sign $\sigma_e$ of each $e\in\EEE_n$ so that it
will receive the exact opposite reward the next time it traverses $e$.
\end{itemize}

A (\textit{dependent}) \textit{random walk at random time} (DWRT/RWRT) with a
nonmonotonic subordinating \textit{random time process} $W_H(n)$ is a process
\setcounter{equation}{6}
\begin{equation}
\label{def-totalreward} \Sm_n=\Sm_n(
\eta_\alpha, W_H):= \sum_{k=1}^{n}
\sum_{e\in\EEE
_k} \sigma_e(k-1)\cdot
\eta_\alpha(e),
\end{equation}
where $\sigma_e(k)\in\{-1,+1\}$ is the sign of $e$ at time $k$.

To explain the name of the process, consider that in an RWRT, due to
cancellation, each reward
$\eta_\alpha(e)$ contributes either one or zero net terms to the sum
(\ref{def-totalreward}).
When $e$ is to the right of the origin, the number of net terms is one
if and only if $W_H(n)$ is to the right of $e$,
and when $e$ is to the left of the origin, the number of net terms is
one if and only if $W_H(n)$ is to the left of $e$. It follows that
%
\begin{equation}
\label{def1} \Sm_n = \sum_{e\in[0,W_H(n)]}
\eta_{\alpha}(e),
\end{equation}
where $e\in[0,x]$ means that $e$ lies between $0$ and $x$ regardless
of the sign of $x$.
The partial sum of reward $\sum_{e\in[0,n]}\eta_{\alpha}(e)$ is
just a random walk $S_\alpha(n)$.
If we let $S_\alpha(0)=0$ and extend the random walk to negative times
in the natural way,
then thinking of time being determined by the location of $W_H(n)$, we have
%
\begin{equation}
\label{def1RW} \Sm_n = S_\alpha\bigl(W_H(n)
\bigr).
\end{equation}
As an aside, if we take (\ref{def1}) as our initial definition rather
than (\ref{def-totalreward}), then the reward
may equally well be placed on the vertices instead of the edges. The
reader may therefore choose to visualize this process
in any of several ways according to his or her own aesthetic preference.

The relationship between $\Sp_n$ and $\Sm_n$ should be clear. In
particular, when the collecting process is a
simple random walk $W(n)$, a relation is made by using a bijection
which assigns to each vertex $k$
either the edge lying to its left whenever the previous step of $W(n)$
was in the positive direction (right),
or the edge lying to its right whenever the previous step of $W(n)$ was
in the negative direction (left).
To extend the relation to other random walks, one must use a modified
version of $\Sp_n$ which,
when going from $x$ to $y$ on the $n$th step, collects a reward not
only from $y$, but
all vertices between $x$ and~$y$. In view of this relationship between
$\Sp_n$ and $\Sm_n$,
if $\P_s$ is the measure for the random scenery, and $\P'$ is the
measure for $W_H$,
then the processes $\Sp_n$ and $\Sm_n$ can be defined on the same
product space with measure $\P_s\times\P'$.

There is a further relationship between $\Sp_n$ and the variations
of $\Sm_n$ which mirrors the connection between BMRS and the
variations of IBM as presented in~\cite{khoshnevisan1999stochastic}.
In order to explain this relationship, it will be convenient to let the
collecting walk $W(n)$ be a simple random walk
and to have the reward for both $\Sp_n$ and $\Sm_n$ be attached to
the edges of $\Z$, rather than to the vertices.
For $p \in\mathbb{N}$, let the \textit{$p$th variation} of $\Sm_n$ be
defined as
%
\begin{equation}
V^{(p)}_n:= \sum_{i=1}^{n}
(\Sm_i - \Sm_{i-1})^p.
\end{equation}

\begin{theorem} \label{rant}
Suppose the i.i.d. reward $\{\eta_{\alpha}(e)\}$ are symmetric and
have finite $p$th moments. If $p$ is odd, then $V^{(p)}_n$ is another
RWRT, while if $p$ is even, then $V^{(p)}_n - n
\E[\eta_\aaa^{p}]$ is a RWRS. In both cases,\vadjust{\goodbreak} the reward collected
by the processes are given by $\{\zeta^{(p)}(e)\}$ where
\[
\zeta^{(p)}(e):= \eta_{\alpha}(e)^p - \E\bigl[
\eta_{\alpha}(e)^p\bigr].
\]
\end{theorem}
\begin{pf}
For $i \geq1$ we let $\EEE_i$ denote the edge between $W(i-1)$ and
$W(i)$. We then have
%
\begin{equation}
\label{rain} \Sp_n = \sum_{i=1}^{n}
\eta_{\alpha}(\EEE_i),\qquad \Sm_n = \sum
_{e\in[0,W(n)]} \eta_{\alpha}(e).
\end{equation}
Note that
%
\begin{equation}
\label{ffff} (\Sm_i - \Sm_{i-1})^p
1_{\{\EEE_i=e\}} = \bigl(\si_{e}(i-1) \eta_\aaa(e)
\bigr)^p 1_{\{\EEE_i=e\}}.
\end{equation}

If $p=2q$ is even, then the sign $ (\si_{e}(i-1) )^{2q}$ in
(\ref{ffff}) is irrelevant. Therefore,
%
\begin{equation}
\label{trid} V^{(2q)}_n - n \E\bigl[\eta_\aaa^{2q}
\bigr]= \Biggl(\sum_{i=1}^{n} (
\Sm_i - \Sm_{i-1})^{2q} \Biggr) - n \E\bigl[
\eta_\aaa^{2q}\bigr] = \sum_{i=1}^{n}
\zeta^{(2q)}(\EEE_i).
\end{equation}
Comparing with (\ref{rain}) shows this to be a RWRS with reward given
by $\{\zeta^{(p)}(e)\}$.

On the other hand, if $p=2q+1$ is odd, then the sign $ (\si_{e(k)}(i-1)
)^{2q+1} = \si_{e(k)}(i-1)$ in (\ref{ffff})
causes the same cancellation as we have with RWRT, and since $\eta_\aaa
$ is symmetric, there is no longer a need to subtract the expectation.
Thus, (\ref{ffff}) yields
%
\begin{equation}
\label{trid2} V^{(2q+1)}_n = \sum
_{i=1}^{n} (\Sm_i -
\Sm_{i-1})^{2q+1} = \sum_{e\in[0,W_H(n)]}
\zeta^{(2q+1)}(e).
\end{equation}
Comparing again with (\ref{rain}) shows this to be a RWRT with reward
given by $\{\zeta^{(p)}(e)\}$.
\end{pf}

We now compare this with the results of \cite
{khoshnevisan1999stochastic}. Let $I_s$ denote an IBM, fix an interval
$[0,t]$, and let
%
\begin{equation}
V_n^{(p)}(t) = \sum_{k=1}^{2^n t}
\bigl(I(T_{k+1,n})-I(T_{k,n})\bigr)^p,
\end{equation}
where $\{T_{k,n}\dvtx 1 \leq k \leq2^n t\}$ is an induced random
partition of the interval $[0,t]$; see~\cite{khoshnevisan1999stochastic},
Section 1, for details. Among other things,
Khoshnevisan and Lewis showed that, when properly renormalized,
$V_n^{(p)}(t)$ converges in distribution to
IBM when $p$ is odd and BMRS when $p$ is even; see Theorems~3.2,
4.4, 4.5 and the discussion in the middle of page 631. If we
consider the natural association between BMRS and RWRS on the one
hand and between IBM and RWRT on the other, we see that the simple
Theorem~\ref{rant} provides an intuitive backdrop for the much
more difficult results concerning the continuous case in
\cite{khoshnevisan1999stochastic}.\vadjust{\goodbreak}

We now return to study of $\Sm_n$ in the general case. We will need
processes extended to noninteger times, and we will therefore
denote the linear interpolation of $\Sm_n$ as
%
\begin{equation}
\label{defconttime} \Sm_t = \Sm_{\lfloor t\rfloor} + \bigl(t-\lfloor
t\rfloor\bigr) (\Sm_{\lceil t\rceil} - \Sm_{\lfloor t\rfloor}).
\end{equation}

Let us now describe the two different \textit{random reward schema} we
will use.
Let us start with an alternating version of the random reward schema
introduced in~\cite{CS}.
Let $\{(W_H^{(i)}(n))_{n\ge0}\}_{i\in\N}$ be independent copies of
$W_H(n)$ which are also independent from independent copies of the
reward $\{\{\eta_\alpha^{(i)}(e)\}_{e\in\Z}\}_{i\in\N}$. If $(c_n)$
is a sequence of integers such that $c_n\to\infty$, then
%
\begin{equation}
\label{defrrs} \sum_{i=1}^{c_n}
\Sm_t\bigl({\eta_\alpha^{(i)}},{W_H^{(i)}}
\bigr)
\end{equation}
is an \textit{alternating
random reward scheme}.

If we instead follow the single-scenery schema of \cite
{dombry14functional} and use
the same single copy of reward $\{\eta_\alpha^{(1)}(e)\}_{e\in\Z}$
for each copy of $W^{(i)}_H(n)$, then
%
\begin{equation}
\label{defrandomrewardschemaSS} \sum_{i=1}^{c_n}
\Sm_t\bigl({\eta_\alpha^{(1)}},{W_H^{(i)}}
\bigr)
\end{equation}
is a \textit{single scenery} alternating random reward scheme.

\subsection{Scaling limits of random reward schema}\label{secrandommeasure}

In this section we state some known results concerning the scalings of
RWRS and DWRS to stable
integral representations. These will motivate our results concerning
the scalings of RWRT and DWRT.

Let us first recall an important definition. Suppose $m$ is a $\sigma
$-finite measure
on a measurable space $(E,\BB)$, and that
\[
\BB_0 = \bigl\{A\in\BB\dvtx m(A)<\infty\bigr\}.
\]

\begin{definition}\label{defrandommeasure}
A \textit{{S$\al$S} random measure} $M$ with \textit{control measure}
$m$ is
a $\sigma$-additive set function on $\BB_0$ such that for all $A_i\in
\BB_0$:
\begin{longlist}[(2)]
\item[(1)] $M(A_1)\sim\SS_{\alpha}(m(A_1)^{1/\alpha})$;
\item[(2)] $M(A_1)$ and $M(A_2)$ are independent whenever $A_1 \cap
A_2 =\varnothing$,
\end{longlist}
where $\SS_\alpha(\sigma)$ is an {S$\al$S} random variable.

In particular, if $f\in L^\alpha(E,\BB,m)$, then
%
\begin{equation}
\label{defint} \int_E f(x) M(dx) \sim\SS_{\alpha}
\bigl(\bigl\|f(x)\bigr\|_{L^\alpha}\bigr).
\end{equation}
\end{definition}

Section 3.3 of~\cite{samorodnitsky1994stable} contains an introduction
to this topic. The immediate importance to us is that the scaling
limits of RWRS and DWRS are integrals with respect to stable random
measures,\vadjust{\goodbreak} where the integral kernel is the local time of a properly
scaled collecting process $\tilde W_{H'}$ (linearly interpolated) which
is either $G_{H'}$ or $S_\beta$ with $\beta\in(1,2]$. The process
$\frac{1}{n^{H'}}\tilde W_{H'}(nt)$ converges\vspace*{-1pt} weakly to a
scaling limit, denoted by ${\tilde {Y}}_t$, which is, respectively,
fBm-$H'$ in $\CC([0,\infty))$ or a $\beta$-stable L\'{e}vy motion in
$\DD([0,\infty))$. Let $(\Omega',\FF',\P')$ be the probability space of
${\tilde{Y}}_t$. It is known that ${\tilde{Y}}_t$ has a jointly
continuous local time $\ell_{\tilde {Y}}(t,x)$; this was shown for
$\beta$-stable L\'{e}vy motions in~\cite{boylan1964local} and for fBm
in~\cite{berman1974}. Moreover, for all $t\ge0$ and all $\alpha\in(0,
2]$, ${\tilde{Y}}_t$ satisfies
%
\begin{equation}
\label{finiteLT} \E' \int_\R\bigl|
\ell_{\tilde{Y}}(t,x)\bigr|^\alpha \,dx <\infty
\end{equation}
by Theorem 3.1 in
\cite{CS} and Lemma 2.1 in~\cite{DG}. Here we interpret
$\ell_{\tilde Y}(t)$ as the increasing family of random functions
which satisfy the occupation time formula
%
\begin{equation}
\label{ltdef} \int_{0}^{t} 1_A(
\tilde Y_s) \,ds = \int_{A} \ell_{\tilde Y}(t,x)
\,dx
\end{equation}
for any Borel set $A$.

Let $M_0(dx)$ be an {S$\al$S} random measure with Lebesgue control
measure which is independent from ${\tilde{Y}}_t$. Throughout this
subsection we will let
%
\begin{equation}
\label{defH} H=1-H'+H'/\alpha.
\end{equation}
A \textit{stable process in random scenery} is an $H$-sssi {S$\al$S} process given
by
%
\begin{equation}
\label{eqnwang} \Delta^H_t(M_0,
\tilde{Y}):=\int_{\R} \ell_{\tilde{Y}}(t,x)
M_0(dx),\qquad t\ge0,
\end{equation}
which is well defined by
(\ref{finiteLT}); see Chapter 3 of~\cite{samorodnitsky1994stable}.
Recall that $\eta_\alpha(k)$ is in the domain of attraction of an
{S$\al$S}
law. It was shown in~\cite{KS,Wang03,CD} that the following weak
convergence holds in $\CC([0,\infty))$:
%
\begin{equation}
\frac{1}{n^{H}}\Sp_{nt}(\eta_\alpha,\tilde
W_{H'}) \Rightarrow\Delta^H_t(M_0,
\tilde{Y}).
\end{equation}
Henceforth we will use $H'$ for the
Hurst parameter of the collecting process and $H$ for the Hurst
parameter of the resulting stable process in random scenery.

The Hurst exponent
$H=1-H'+H'/\alpha$
can be explained by using the local time scaling relation
%
\begin{equation}
\bigl(\ell_{\tilde{Y}}(ct,x),x\in\R,t\ge0 \bigr)\ed\bigl(c^{1-H'}
\ell_{\tilde{Y}}\bigl(t,x/c^{H'}\bigr),x\in\R,t\ge0 \bigr).
\end{equation}

In~\cite{CS}, weak convergence in $\CC([0,\infty))$ was shown for a
properly normalized random reward scheme
\[
c_n^{-1/\alpha}\sum
_{i=1}^{c_n} {n}^{-({\alpha+1})/({2\alpha})}Z_{nt}
\bigl({\eta_\alpha^{(i)}},{S_2^{(i)}}\bigr),
\]
where $Z_t$ is the linear interpolation of $\Sp_n$ in the same
manner as (\ref{defconttime}). The
\[
\bigl\{S_2^{(i)}(n)\bigr\}_{i\in\N}
\]
are
independent copies of mean zero, finite variance $(\beta=2)$ random
walks which have $H'=1/2$ explaining the exponent
$H=\frac{\alpha+1}{2\alpha}$. They collect independent copies of i.i.d.
reward $\{\eta_\alpha^{(i)}(k)\}_{i\in\N}$ which are also independent
from the random walks. Cohen and Samorodnitsky called the limiting
process an fBm-$1/2$ local time fractional stable motion. In
\cite{CD}, the discrete collecting process was generalized to
$G_{H'}$ and convergence to fBm-$H'$ local time fractional stable
motions for any $H'\in(0,1)$ was proved. In~\cite{DG}, a collecting
process scaling to $\beta$-stable L\'{e}vy motion ($\beta>1$) was used,
and consequently, other local time fractional stable motions were
obtained in the limit. Let us now explicitly state these collective
results.

Recall that $(\Omega',\FF',\P')$ is the probability space of
${\tilde{Y}}_t$.
Suppose $M_1(d\omega',dx)$ is an {S$\al$S} random measure that has
control measure $\P'\times\mbox{Lebesgue}$,
but lives on some other probability space $(\Omega,\F,\P)$. As
above, $\tilde W_{H'}$ is either $G_{H'}$ or $S_\beta$ with $\beta\in(1,2]$.
Letting $H$ be as in (\ref{defH}), in light of (\ref{finiteLT}) we
define a \textit{local time fractional stable motion} as the process
%
\begin{equation}
\label{defGamma} \Gamma^H_t(M_1,
\tilde{Y}):= \int_{\Om'\times\R} \ell_{\tilde
{Y}}\bigl(t,x;
\omega'\bigr) M_1\bigl(d\omega', dx\bigr),\qquad t
\ge0.
\end{equation}

Let $(c_n)$ be an integer sequence with $c_n\to\infty$, and let
$\{\eta_\alpha^{(i)}(k)\}$ be independent copies of i.i.d. reward in
the domain of attraction of an {S$\al$S} law.
The following weak convergence holds in $\CC([0,\infty))$ as $n\to
\infty$:
%
\begin{equation}\label{is0}\quad
c_n^{-1/\alpha}\sum_{i=1}^{c_n}
\frac{1}{n^{H}} Z_{nt}\bigl({\eta_\alpha^{(i)}},{
\tilde W^{(i)}_{H'}}\bigr) \Rightarrow\Gamma^H_t(M_1,
\tilde{Y}) \qquad\mbox{(independent scenery).}
\end{equation}

Let $M_2$ be a stable random measure with Lebesgue control measure with
the restriction that $\alpha\in(1,2]$,
and again let $H$ be as in (\ref{defH}). We may use (\ref{finiteLT})
and H\"{o}lder's inequality to define
%
\begin{equation}
\label{defLambda} \Lambda^H_t(M_2,
\tilde{Y}):= \int_\R\E' \ell_{\tilde
{Y}}
\bigl(t,x;\omega'\bigr) M_2(dx),\qquad t\ge0.
\end{equation}
Note that the scale parameter at time $t$ for (\ref{defLambda}) is
%
\begin{equation}
\label{scalepar} \sigma= \bigl\|\E' \ell_{\tilde{Y}}\bigl(t,x;
\omega'\bigr)\bigr\|_{L^\alpha(\R)}
\end{equation}
versus
$\sigma= \|\ell_{\tilde{Y}}(t,x;\omega')\|_{L^\alpha(\Omega
'\times\R)}$ for (\ref{defGamma}).
For $\alpha\in(1,2]$, a convergence result (in finite-dimensional
distributions) with respect to the single scenery case
was given in Theorem 4.2 of~\cite{dombry14functional}:
%
\begin{equation}\label{ss0}
c_n^{-1}\sum_{i=1}^{c_n}
\frac{1}{n^{H}} Z_{nt}\bigl({\eta_\alpha^{(1)}},{
\tilde W_{H'}^{(i)}}\bigr) \stackrel{\mathrm{f.d.d.}}{\lar}\Lambda^H_t(M_2,
\tilde{Y}) \qquad\mbox{(single scenery)}.
\end{equation}
As stated earlier, the process on the right-hand side is $H$-ss, but
using (\ref{scalepar}) one can see that
this process does not in general have stationary increments.

It is convenient to
write (\ref{defGamma}) and (\ref{defLambda}) as renormalized sums
of (\ref{eqnwang}) which appeal to
the stable central limit theorem and the law of large numbers,
respectively; see~\cite{CD,DG,dombry14functional}.
The former renormalization is applied to the entire integral in (\ref
{eqnwang}), and the convergence is in $\CC([0,\infty))$
whereas the latter renormalization applies only to the integral kernel
%
\begin{eqnarray}
\label{resultCLT}n^{-1/\alpha}\sum_{i=1}^n \bigl(
\Delta^H_t\bigr)^{(i)} &\Rightarrow&
\Gamma^H_t,
\\
\label{resultLLN}
\int_\R\Biggl(n^{-1}\sum
_{i=1}^n \ell_{\tilde{Y}}^{(i)}(t,x)
\Biggr) M_2(dx) &\stackrel{\mathrm{f.d.d.}}{\lar}& \Lambda^H_t.
\end{eqnarray}

\subsection{Scaling limits of alternating random reward schema}

We are now ready to state our results concerning the scaling limits of
$\Sm(t)$ and its associated random reward schema
(\ref{defrrs}) and (\ref{defrandomrewardschemaSS}).

Throughout this subsection we assume that
the discrete collecting process $W_{H'}(n)$ is extended to continuous
time by linear interpolation and
that it has the scaling limit ${Y}_t$ as given in condition (\ref{Wgamma}).
Independent copies of i.i.d. reward $\{\eta_\alpha^{(i)}(k)\}_{i\in
\N}$ are, as usual, in the domain of attraction
of an {S$\al$S} law (scale parameter $\sigma=1$) and independent
from the
random walks.
The space $(\Omega',\F',\P')$ supports ${Y}_t$, and the {S$\al$S}
random measures $M_i$ are as in the previous subsection.
Define the processes
%
\begin{eqnarray}
\label{defDelta2}\qquad
\Delta_H(t) = \Delta_H(t; M_0,Y) &:=& \int
_\R1_{[0,{Y}_t(\omega
')]}(x) M_0(dx),\qquad t\ge0,
\\
\label{defGamma2}
\Gamma_H(t) = \Gamma_H(t; M_1,Y)&:=& \int
_{\Om'\times\R} 1_{[0,{Y}_t(\omega')]}(x) M_1\bigl(d
\omega', dx\bigr),\qquad t\ge0,
\\
\label{defLambda2}
\Lambda_H(t) = \Lambda_H(t; M_2,Y)&:=& \int
_\R\E' 1_{[0,{Y}_t(\omega')]}(x)
M_2(dx),\qquad t\ge0,
\end{eqnarray}
which are analogous to (\ref{eqnwang}), (\ref{defGamma}) and (\ref
{defLambda}).

The above are all self-similar with common index $H=H'/\alpha$, and
(\ref{defGamma2}) and (\ref{defLambda2}) are {S$\al$S} processes. One
can also observe (see Theorem 2.2 in~\cite{jung2010indicator}) that
both (\ref{defDelta2}) and (\ref{defGamma2}) have stationary
increments. We call (\ref{defDelta2}) an (\textit{$H'$-sssi})\textit{-time
$\alpha$-stable L\'{e}vy motion} or more generally a \textit{stable
process at random time}. If $X(t)$ is a two-sided $\alpha$-stable
L\'{e}vy motion, then we may also write (\ref{defDelta2}) as
$X(Y_t)$. The process (\ref{defGamma2}) is an \textit{indicator
fractional stable motion} as introduced in~\cite{jung2010indicator}.
The process (\ref{defLambda2}) is the alternating analog of the
scaling limit of a single scenery random reward scheme introduced in
\cite{dombry14functional}.
%
\begin{theorem} \label{t1}
Let $H=H'/\alpha$, and let $c_n\to\infty$ as $n \to\infty$.
\begin{itemize}
\item The following convergence holds in f.d.d.:
%
\begin{equation}
\label{s} n^{-H} S_\alpha\bigl(W_{H'}(nt)\bigr)
\Rightarrow \Delta_H(t; M_0,Y).
\end{equation}
If the reward are symmetric with finite variance ($\alpha=2$), and
$n^{-H'}W_{H'}(nt)$ converges weakly in $\DD([0,\infty))$ ($\CC
([0,\infty))$), then (\ref{s}) also holds weakly in $\DD([0,\infty
))$ ($\CC([0,\infty))$,
resp.).
\item
If $n^{-H}|W_H(\lfloor nt\rfloor)|$ is uniformly integrable, then
%
\begin{eqnarray}\label{is}
&&
c_n^{-1/\alpha}\sum_{i=1}^{c_n}
n^{-H} \Sm_{nt}\bigl({\eta_\alpha^{(i)}},{W_{H'}^{(i)}}
\bigr) \nonumber\\[-8pt]\\[-8pt]
&&\qquad\stackrel{\mathit{f.d.d.}}{\lar} \Gamma_H(t; M_1,Y) \qquad\mbox{(independent
scenery)}.\nonumber
\end{eqnarray}
\item
If $\alpha> 1$, then
%
\begin{equation}\label{ss}
c_n^{-1}\sum_{i=1}^{c_n}
n^{-H} \Sm_{nt}\bigl({\eta_\alpha^{(1)}},{W_{H'}^{(i)}}
\bigr) \stackrel{\mathit{f.d.d.}}{\lar} \Lambda_H(t; M_2,Y) \qquad\mbox{(single
scenery)}.
\end{equation}
\end{itemize}
\end{theorem}
The interest of the first convergence result [to ($H'$-sssi)-time
$\alpha$-stable L\'{e}vy motion] lies in the fact that this seems to be
the first such Donsker-type theorem for iterated processes where the
random time process is not a subordinator, that is, not an increasing
L\'{e}vy process. In the case where the random time process is a
subordinator, similar convergence results are well known. In fact,
in Section 2.2 of~\cite{nane2011local}, such results are extended to
the case where the scenery have a certain dependence structure.
Their Donsker-type theorem shows convergence to an
$\alpha$-time fractional Brownian motion.

It is not hard to see that $\Delta_H(t), \Gamma_H(t)$, and
$\Lambda_H(t)$ are all continuous in probability. However, by
Theorem 10.3.1 in~\cite{samorodnitsky1994stable}, when $\alpha<2$,
$\Delta_H(t)$ and $\Gamma_H(t)$ are not sample continuous. In those
cases, the best we can hope for is weak convergence in
$\DD([0,\infty))$. We will see in the remark at the end of Section
\ref{secproofs}, that even this is a lot to ask. In that
remark, it is argued that even in the simplest cases, $\Delta_H(t)$ is
not even
in $\DD([0,\infty))$. In particular, the weak convergence in
$\CC([0,\infty))$ and $\DD([0,\infty))$ given in the first part of Theorem
\ref{t1} depends heavily on the fact that $\alpha=2$. In this case,
the scaling limit of $S_\alpha$ is continuous since it is simply
Brownian motion.

The condition that $n^{-H'}W_{H'}(\lfloor nt\rfloor)$ is uniformly
integrable holds when $W_{H'}$ is either $G_{H'}$ or $S_\beta$, $\beta>1$.
The former follows from a Gaussian concentration inequality which bounds
$n^{-H'}W_{H'}(\lfloor nt\rfloor)$ in $L^p$ for all $p\ge1$ (see
\cite{ledoux1991probability}, page 60),
and the latter follows from equation (5.s) in~\cite{legall1991range}
and the bound $\E(|X|1_A)\le\|X\|_{p}(\P(A))^{1/q}$.

\section{\texorpdfstring{Proof of Theorem \protect\ref{t1}}{Proof of Theorem 2.3}}\label{secproofs}

A convenient tool in proving convergence of the finite-dimensional
distributions is a diagonal convergence theorem of~\cite{dombrySWN}.
In order to state this theorem, we require some definitions.\vadjust{\goodbreak}

As usual $\eta_\alpha(k)$ is in the domain of attraction of the
{S$\al$S}
law with scale parameter $\sigma=1$, and it is the reward on the
edge between $k$ and $k+1$. For fixed positive $h$, define $\mu_h$
to be the random signed measure on $\R$ which is a.s. absolutely
continuous with respect to Lebesgue measure and whose random density
is given by
%
\begin{equation}
\frac{d\mu_h}{dx}(x)= h^{-1+1/\alpha}\sum_{k\in\Z}
\eta_\alpha(k) 1_{(hk,h(k+1)]}(x).
\end{equation}
For a locally integrable
function $f\in L^1_{\mathrm{loc}}$, define
%
\begin{equation}
\mu_h[f] = \int f \,d\mu_h:= \sum
_{k\in\Z} \eta_\alpha(k)h^{-1+1/\alpha}\int
_{hk}^{h(k+1)} f(x) \,dx.
\end{equation}
For $0<\alpha<1$, we will say that $(f_n)_{n\in\N}$
converges to $f$ in $\DD^\alpha$ if the following two conditions
hold:
\begin{itemize}
\item for any compact $K\subset\R$, $f_n{\bf1}_K$ converges to
$f{\bf1}_K$ in $L^1(\R)$;
\item there is some $\eta>\alpha^{-1}$ such that $f_n(x)=o(|x|^{-\eta
} )$ and $f(x)=o(|x|^{-\eta})$ as $x\to\infty$.
\end{itemize}

Let $\FF^\alpha=L^\alpha(\R)$ if $1\leq\alpha\leq2$ and $\FF^\alpha=\DD
^\alpha$ if $0<\alpha<1$. The following diagonal
convergence is shown in
Proposition~3.1 of~\cite{dombrySWN}; see also Proposition~3.1 of~\cite
{dombry14functional}.
%
\begin{prop}[(Dombry)]\label{propdombry}
Suppose $M_0(dx)$ is an $\alpha$-stable random measure, $\alpha\in
(0,2]$ and $(f_n)_{n\in\N}$ converges to $f$ in $\FF^\alpha$.
If $h_n\to0$ as \mbox{$n\to\infty$}, then the random variables $\mu
_{h_n}[f_n]$ converge weakly as \mbox{$n\to\infty$} and in particular,
%
\begin{equation}
\label{diagconv} \mu_{h_n}[f_n] \Rightarrow
\int_\R f M_0(dx).
\end{equation}
\end{prop}

We now start by showing convergence in f.d.d. for Theorem~\ref{t1}.
However, to reduce notation and simplify the presentation,
we only prove convergence of the one-dimensional distributions for some
fixed $t>0$.
The extension to f.d.d. in all three cases follows easily using the
Cram\'er--Wold device; see, e.g., Theorem 3.9.5 in~\cite{durrett2010probability}.

Also without loss of generality we use $n^{-H}A_{\lfloor nt\rfloor}$
instead of the linear
interpolation $n^{-H}A_{nt}$ since they differ by at most $n^{-H}\eta
_\alpha(k)$ which goes a.s. to $0$ as $n\to\infty$.

\subsection*{Convergence in f.d.d. for (\protect\ref{s})}
Fix $t\in[0,\infty)$. Let $X_n(t)=\frac{1}{n^{H'}}W_{H'}(\lfloor
nt\rfloor)$.
According to assumption (\ref{Wgamma}), $X_n(t) \Rightarrow{Y}(t)$. By
Skorohod's representation theorem, there is a common probability space
on which $\bar{X}_n\ed X_n(t)$, $\bar{Y}\ed Y(t)$ live and such that
$\bar{X}_n(\bar\omega)\to\bar{Y}(\bar\omega)$ for all $\bar
\omega\in\bar\Om$ (note that the bar includes the dependence on $t$).\vadjust{\goodbreak}

Fix an $\bar\omega$ and recall that $H=H'/\alpha$ and that for
$a<0$, we let $[0,a]:=[a,0]$. We have
%
\begin{eqnarray}
\label{unraveldefs}
&&
\mu_{n^{-H'}}[1_{[0,\bar{X}_n({\bar\omega})]}]
\nonumber\\
&&\qquad=
n^{H'- H'/\alpha}\sum_{k\in\Z}
\eta_\alpha(k) \int_{kn^{-H'}}^{(k+1)n^{-H'}}1_{[0,\bar{X}_n({\bar\omega})]}(x)
\,dx
\nonumber\\[-8pt]\\[-8pt]
&&\qquad= n^{H'- H'/\alpha}\sum_{k\in\Z} \eta_\alpha(k)
n^{-H'} 1_{\{
\bar{W}_{H'}(\lfloor nt\rfloor)>k\ge0\}\cup\{\bar{W}_{H'}(\lfloor
nt\rfloor)\le k<0\}}(\bar\omega)
\nonumber\\
&&\qquad= n^{-H}A_{\lfloor nt\rfloor}\bigl(\eta_\alpha,
\bar{W}_{H'}({\bar\omega})\bigr).\nonumber
\end{eqnarray}
By Proposition~\ref{propdombry} and the fact that
$
1_{[0,\bar{X}_n(\bar\omega)]} \to1_{[0,\bar{Y}(\bar\omega)]}
$ in $\FF^\alpha$,
we have that the one-dimensional distributions of
$n^{-H} \Sm_{\lfloor nt\rfloor}(\eta_\alpha,W_{H'})$ converge to
those of $\Delta_H(t; M_0,Y)$.

\subsection*{Convergence in f.d.d. for (\protect\ref{ss})}
For multiple independent walkers in the same scenery, we follow the
arguments of Proposition 2.4 in~\cite{dombry14functional}.
Fix $t\in[0,\infty)$. As in the proof of (\ref{s}), using Skorohod's
representation theorem and Proposition~\ref{propdombry},
we have for $\alpha\in(1,2]$,
\[
\frac{1}{c_n}\sum_{i=1}^{c_n}n^{-H'/\alpha}
\Sm_{\lfloor nt\rfloor
}\bigl(\eta_\alpha,\bar{W}_{H'}^{(i)}
\bigr) = \mu_{n^{-H'}}\Biggl[c_n^{-1}\sum
_{i=1}^{c_n}1_{[0,\bar{X}^{(i)}_n(\bar
\omega)]}\Biggr],
\]
where $\bar{X}_n^{(i)}(\bar\omega)\to\bar{Y}^{(i)}(\bar\omega)$
for each $i\in\N$ and for all $\bar\omega\in\bar\Om$ (the bar
includes the dependence on $t$).

We need only show the following converges in probability to zero as
\mbox{$n\to\infty$}:
%
\begin{equation}
\Biggl\|\frac{1}{c_n}\sum_{i=1}^{c_n}
(1_{[0,\bar
{X}^{(i)}_n(\bar\omega)]} ) -1_{[0,\bar{Y}^{(i)}(\bar\omega
)]}+1_{[0,\bar{Y}^{(i)}(\bar\omega
)]}- \E'1_{[0,{Y}_{t}]}
\Biggr\|_{L^\alpha(\R)},
\end{equation}
where for fixed $n$, the random variables $\bar{X}^{(i)}_n, 1\le i\le
c_n$ are i.i.d.
Also, for each fixed $i$, $\bar{X}^{(i)}_n$ converges a.s. to $\bar{Y}^{(i)}$.
We first show that as $n\to\infty$,
%
\begin{eqnarray}
\label{triangularWLLN} &&\Biggl\|\frac{1}{c_n}\sum
_{i=1}^{c_n} (1_{[0,\bar
{X}^{(i)}_n(\bar\omega)]}-1_{[0,\bar{Y}^{(i)}(\bar\omega)]} )
\Biggr\|_{L^\alpha(\R)}
\nonumber\\[-8pt]\\[-8pt]
&&\qquad\le\frac{1}{c_n}\sum_{i=1}^{c_n}
\| 1_{[0,\bar
{X}_n^{(i)}(\bar\omega)]}-1_{[0,\bar{Y}^{(i)}(\bar\omega)]}
\|_{L^\alpha(\R)}\cip0.\nonumber
\end{eqnarray}
Consider a triangular array such that for each fixed $n$, there are
$c_n$ i.i.d. random variables
\[
\bigl(U^{(n)}_i \bigr)_{1\le i\le c_n}:= \bigl(\|
1_{[0,\bar
{X}_n^{(i)}(\bar\omega)]}-1_{[0,\bar{Y}^{(i)}(\bar\omega)]}
\|_{L^\alpha(\R)}
\bigr)_{1\le i\le c_n}
\]
in each row, and for each fixed $i$, the column of random variables
$(U_i^{(n)})_{n\in\N}$ converges weakly to zero. For such triangular arrays,
the following weak law holds (see Proposition 2.4 in \cite
{dombry14functional}):
%
\begin{equation}
\label{Weaklaw} \frac{1}{c_n}\sum_{i=1}^{c_n}
U_i^{(n)} \cip0 \qquad\mbox{as } n\to\infty,
\end{equation}
thus proving (\ref{triangularWLLN}).

Since $\E\|1_{[0,{Y}_{t}]}\|_{L^\alpha(\R)}<\infty$, the strong law
of large numbers for Banach space valued random
variables implies that the following converges a.s. in~$L^\alpha(\R)$:
%
\begin{equation}
\frac{1}{c_n}\sum_{i=1}^{c_n}1_{[0,\bar{Y}^{(i)}(\bar\omega)]}
\to\E'1_{[0,{Y}_{t}]},
\end{equation}
thus completing the proof of one-dimensional weak convergence for
(\ref{ss}).

\subsection*{Convergence in f.d.d. for (\protect\ref{is})}
We will mimic the arguments of~\mbox{\cite{KS,DG,CD}}.
Let
\[
1_{\{W_{H'}(nt;k)\}}\bigl(\omega'\bigr):=1_{\{W_{H'}(\lfloor nt\rfloor
)>k\ge0\}
\cup\{W_{H'}(\lfloor nt\rfloor)\le k<0\}}\bigl(
\omega'\bigr).
\]
Using the last equality in (\ref{unraveldefs}), we have
%
\begin{eqnarray}
\label{eq40}
&&\E\exp\Biggl(\sum_{\ell=1}^{c_n}i
\th c_n^{-1/\alpha}n^{-H} \Sm_{\lfloor nt\rfloor}\bigl({
\eta_\alpha^{(\ell)}},{W_{H'}^{(\ell
)}}\bigr)
\Biggr)
\nonumber\\
&&\qquad=
\E\exp\Biggl(\sum_{\ell=1}^{c_n}i
\th c_n^{-1/\alpha} n^{-H}\sum
_{k\in\Z} \eta^{(\ell)}_\alpha(k) 1_{\{W_{H'}^{(\ell
)}(nt;k)\}}
\bigl(\omega'\bigr) \Biggr)
\\
&&\qquad= \biggl(\E' \biggl[ \prod_{k\in\Z}
\phi_{\eta_\alpha}\bigl(v_n\bigl(\omega',k\bigr) \bigr)
\biggr] \biggr)^{c_n},\nonumber
\end{eqnarray}
where $\phi_{\eta_\alpha}$ is the real-valued characteristic
function of a symmetric reward $\eta_\alpha$ and
%
\begin{equation}
v_{n}\bigl(\omega',k\bigr)=\th c_n^{-1/\alpha}
n^{-H} 1_{\{W_{H'}(nt;k)\}
}\bigl(\omega'\bigr),\qquad k\in\Z.
\end{equation}

Suppose $\phi_\alpha(v)=\exp(-|v|^\alpha)$ is the characteristic
function of the {S$\al$S} law of scale parameter $\sigma=1$.
We show that the following asymptotic holds as $n\rightarrow\infty$:
%
\begin{equation}
\label{cd1} \E' \biggl[ \prod_{k\in\Z}
\phi_{\eta_\alpha}\bigl(v_n\bigl(\omega',k\bigr) \bigr)
\biggr] = \E' \biggl[ \prod_{k\in\Z}
\phi_{\alpha}\bigl(v_n\bigl(\omega',k\bigr) \bigr)
\biggr] +o\bigl(c_n^{-1}\bigr).
\end{equation}
If $(x_i)_{i\in\Z}$ and $(x_i')_{i\in\Z}$ are sequences in $[-1,1]$
with only finitely many terms not equal to one, then
%
\begin{equation}
\biggl\llvert\prod_{i\in\Z}x'_i
-\prod_{i\in\Z}x_i\biggr\rrvert\leq\sum
_{i\in\Z} \bigl\llvert x'_i -
x_i\bigr\rrvert.\vadjust{\goodbreak}
\end{equation}
Letting
%
\begin{equation}
g(y)=\sup_{|x|\leq y} |x|^{-\alpha}\bigl\llvert\phi_{\eta_\alpha}(x)-
\phi_{\alpha}(x)\bigr\rrvert,\qquad x\neq0,
\end{equation}
we have
%
\begin{eqnarray}\label{cd2}
&&c_n\biggl\llvert\prod_{k\in\Z}
\phi_{\eta_\alpha} \bigl( v_n\bigl(\omega',k\bigr)
\bigr)- \prod_{k\in\Z} \phi_\alpha\bigl(
v_n\bigl(\omega',k\bigr) \bigr) \biggr\rrvert
\nonumber
\\
&&\qquad\leq c_n \sum_{k\in\Z} \bigl\llvert
\phi_{\eta_\alpha} \bigl( v_n\bigl(\omega',k\bigr)
\bigr)- \phi_{\alpha} \bigl( v_n\bigl(\omega',k
\bigr) \bigr) \bigr\rrvert
\nonumber
\\
&&\qquad\leq g\Bigl(\sup_{k\in\Z} \bigl|v_n\bigl(\omega',k
\bigr)\bigr|\Bigr) \sum_{k\in\Z}c_n
\bigl|v_n\bigl(\omega',k\bigr)\bigr|^\alpha
\\
&&\qquad= g\bigl(\th c_n^{-1/\alpha} n^{-H}\bigr) \sum
_{k\in\Z} \bigl\llvert n^{-H} \th
1_{\{W_{H'}(nt;k)\}}\bigl(\omega'\bigr)\bigr\rrvert^\alpha
\nonumber
\\
&&\qquad= g\bigl(\th c_n^{-1/\alpha} n^{-H}\bigr) |
\th|^\alpha n^{-H'} \bigl|W_{H'}\bigl(\lfloor nt\rfloor;
\omega'\bigr)\bigr|.\nonumber
\end{eqnarray}

By assumption, $n^{-H'}|W_{H'}(\lfloor nt\rfloor;\omega')|$ converges
weakly and is bounded in~$L^1$, so to prove
(\ref{cd1}) we need only show that $g(\th c_n^{-1/\alpha} n^{-H})$ is
bounded and converges in probability to $0$.
Since $\eta_\alpha$ is in the domain of attraction of the {S$\al$S} law
with $\sigma=1$,
by the stable central limit theorem, we have that as $v\to0$,
\[
\phi_{\eta_\alpha}(v) = \phi_{\alpha}(v)+o\bigl(|v|^\alpha
\bigr).
\]
Thus $g$ is bounded, continuous and vanishes at $0$. Equation (\ref
{cd1}) follows since $\th c_n^{-1/\alpha} n^{-H}$ goes to zero.

Let $\{(\zeta_\alpha^{(\ell)}(k))_{k\in\Z}\}_{\ell\in\N}$ be
independent copies of i.i.d. reward such that $\zeta_\alpha^{(0)}(1)$
has an {S$\al$S} law with
scale parameter $\sigma=1$. Using (\ref{cd1}), the $c_n$th root of
(\ref{eq40}) is equal to
%
\begin{eqnarray}
&&\E\exp\bigl(i\th c_n^{-1/\alpha}n^{-H}
\Sm_{\lfloor nt\rfloor
}\bigl({\eta_\alpha^{(1)}},{W_{H'}^{(1)}}
\bigr) \bigr)
\nonumber
\\
&&\qquad=
\E\exp\bigl(i\th c_n^{-1/\alpha}n^{-H}
\Sm_{\lfloor
nt\rfloor}\bigl({\zeta_\alpha^{(1)}},{W_{H'}^{(1)}}
\bigr) \bigr)+o\bigl(c_n^{-1}\bigr)
\nonumber\\
&&\qquad=
\E'\exp\biggl(-c_n^{-1}n^{-H'}
\sum_{k\in\Z}\bigl(\theta1_{\{
W_{H'}(nt;k)\}}\bigl(
\omega'\bigr)\bigr)^\alpha\biggr)+o\bigl(c_n^{-1}
\bigr)
\\
&&\qquad=
\E'\exp\bigl(-c_n^{-1}n^{-H'}\bigl|W_{H'}
\bigl(\lfloor nt\rfloor;\omega'\bigr)\bigr|\theta^\alpha\bigr)+o
\bigl(c_n^{-1}\bigr)
\nonumber\\
&&\qquad= \E' \bigl(1-c_n^{-1}n^{-H'}\bigl|W_{H'}
\bigl(\lfloor nt\rfloor;\omega'\bigr)\bigr|\theta^\alpha+o
\bigl(c_n^{-1}\bigr) \bigr).\nonumber
\end{eqnarray}

If $b_n$ is such that $c_nb_n\to\lambda$, then $(1+b_n)^{c_n}\to
e^\lambda$.
Letting
\[
b_n= -c_n^{-1}\E'
\bigl(n^{-H'}\bigl|W_{H}\bigl(\lfloor nt\rfloor\bigr)\bigr|\bigr)
\theta^\alpha
\]
and using the assumption of uniform integrability, we have that
%
\begin{equation}
\bigl(1-c_n^{-1}\E'\bigl(n^{-H'}\bigl|W_{H}
\bigl(\lfloor nt\rfloor\bigr)\bigr|\bigr)\theta^\alpha+o\bigl(c_n^{-1}
\bigr) \bigr)^{c_n}\to e^{-|\theta|^\alpha\E'|Y_t|}
\end{equation}
as required.

\subsection*{Tightness in $\DD([0,\infty))$ and $\CC([0,\infty))$
for (\protect\ref{s})}

Suppose that $\alpha=2$ so that $n^{-1/\alpha}S_\alpha(n^{1/\alpha
}t)$ converges weakly in $\CC(\R)$ to a two-sided Brownian motion
$\tilde{B}_t$.
By (\ref{Wgamma}) and the independence of $S_\alpha(t)$
and $W_{H'}(t)$, the joint process
\[
\bigl(n^{-H'}S_\alpha\bigl(n^{H'}t
\bigr),n^{-H'}W_{H'}(nt)\bigr)
\]
converges weakly to $(\tilde{B}_t,Y_t)$ in $\CC(\R) \times\DD
([0,\infty))$. The weak convergence of $n^{-H'}S_\alpha(W_{H'}(nt))$
in $\DD([0,\infty))$ therefore follows from the continuous mapping
theorem, provided that $(x,y) \longrightarrow x \circ y$ is continuous
from $\CC(\reals) \times\DD([0,\infty))$ to $\DD([0,\infty))$.

The topologies on $\CC(\R)$ and $\DD([0,\infty))$ are first
countable, so proving sequential continuity suffices. Suppose $x_n
\longrightarrow x$ in $\CC(\reals)$ and $y_n \longrightarrow y$ in
$\DD([0,\infty))$, and let $T$ be a continuity point of $y$. By the
definition of convergence on $\DD([0,\infty))$, we must show that
there is a sequence of homeomorphisms $\lambda^T_n$ from $[0,T]$ onto
$[0,T]$ such that $\lambda^T_n$ converges uniformly to the identity
and $x_n \circ y_n \circ\lambda^T_n$ converges uniformly to $x \circ y$.

Let $\eps> 0$ be given. Since $y_n \to y$ in $\DD([0,\infty))$,
there are homeomorphisms $\lambda^T_n$ from $[0,T]$ onto $[0,T]$ such
that $\lambda^T_n$ converges uniformly to the identity, and $y_n \circ
\lambda^T_n$ converges uniformly to $y$. The set $\AA=\bigcup_n
y_n([0,T])$ is bounded, so $x_n$ converges uniformly to $x$ on $\bar
\AA$. Thus, $x$ is uniformly continuous on $\bar\AA$, and thus on
$\AA$.

Choose\vspace*{1pt} $\delta>0$ such that $|x(y_1)-x(y_2)| < \frac{\eps}{2}$ for
$y_1, y_2 \in\AA$ and $|y_1-y_2|<\delta$. Next, find $M_1>0$ such
that $\sup_{t \in[0,T]} |y_n \circ\lambda_n^H(t)-y(t)|< \delta$
whenever $n>M_1$, and find $M_2>0$ such that $\sup_{y \in\AA}
|x_n(y) - x(y)| < \frac{\eps}{2}$ whenever $n>M_2$. Then, whenever
$n>\max(M_1,M_2)$, we have
%
\begin{eqnarray}
&&\bigl|x_n \circ y_n \circ\lambda_n^H(t)
- x \circ y(t)\bigr|
\nonumber
\\
&&\qquad \leq \bigl|x_n \circ y_n \circ
\lambda_n^H(t) - x \circ y_n \circ
\lambda_n^H(t)\bigr| + \bigl|x \circ y_n \circ
\lambda_n^H(t) - x \circ y(t)\bigr|
\\
&&\qquad \leq \frac{\eps}{2} + \frac{\eps}{2} = \eps\nonumber
\end{eqnarray}
for all $t \in[0,T]$. Thus, $x_n \circ y_n \circ\lambda_n^H$
converges uniformly to $x \circ y$ showing continuity of the
composition map. The same argument holds if $n^{-H'}W_{H'}(nt)$
converges weakly in $\CC([0,\infty))$, except that proving the
continuity of the composition map on $\CC(\reals) \times\CC
([0,\infty))$ is even simpler.
\begin{remark*}
We thank an anonymous referee for the above tightness proof which
simplifies our original proof. The referee also noticed the following
informative observation.\vadjust{\goodbreak} If $\alpha<2$, then $S_\alpha$ scales to an
$\alpha$-stable L\'{e}vy motion, $X(t)$. Fix $\eps>0$ and let $\tau
_{\eps}>0$ be the first positive time such that $|X(\tau_\eps)-\lim
_{t\to\tau_\eps^-}X(t)|>\eps$. Consider the simple case where
$W_{H'}$ scales to a Brownian motion, $B_t$. Let $\tau$ be the first
time $B_t-\tau_\eps$ hits $0$. As is well known, $B_t-\tau_\eps$
oscillates around $0$ immediately, thus $\lim_{t\to\tau^+}X(B_t)$
does not exist a.s. This argument, which can be made rigorous, shows
that even in the elementary case where the collecting process
scales to Brownian motion, the process $X(B_t)$ is not cadlag.
\end{remark*}

\section{A recursive construction of some fractional stable
motions}\label{secconstruction}
Throughout this section we will suppose that $\alpha\in(1,2]$. We
present two related recursive constructions of some $H$-sssi
processes. The first recursion produces stable processes in random
scenery, while the second recursion produces local time and
indicator fractional stable motions. Note that only the second
recursion leads to {S$\al$S} processes. Since fBm is the only sssi
Gaussian process, when $\alpha=2$ the second construction gives us fBm.
In particular, if on the first step of the recursion we use Brownian
motion as the collecting process (or random time process), then we
obtain fBm of any dyadic Hurst parameter.

Although the first construction does not in general lead to
$\alpha$-stable processes, we will see that the finite-dimensional
distributions of the processes have finite $\alpha$ moments, and thus one
can appeal to the stable central limit theorem and normalize partial
sums of independent copies of the stable processes in random scenery
in order to get honest stable processes [in a manner similar to (\ref
{resultCLT})].

Let $Y^\varnothing_t$ be an $H$-sssi process satisfying the four
conditions of Theorem~\ref{pp} below. Consider the vector ${v}=(v_1,
\ldots,v_n)$ with coordinates $v_j \in\{+,-\}$. Let us use the
notation $\hat{{{v}}}$ to denote ${{v}}$ truncated by removing the
last element, that is, $\hat{{{v}}}=(v_1, \ldots,v_{n-1})$. The
empty set $\varnothing$ will denote the empty vector.

We define the process $Y^{v}_t$ recursively from $Y^{\hat{{{v}}}}_t$
and an $\alpha$-stable random measure $M_0(dx)$, with $\alpha\in(1,2]$,
assumed to be independent from $Y^{\hat{v}}_t$. If $v_n=(+)$, we let
%
\begin{equation}
\label{yp} Y^{v}_t:= \int_{\R}
\ell_{Y^{\hat
{v}}}(t,x) M_0(dx),
\end{equation}
and if $v_n=(-)$, we let
%
\begin{equation}
\label{ym} Y^{v}_t:= \int_{\R}
1_{[0,Y^{\hat{v}}_t]}(x) M_0(dx).
\end{equation}

The second recursive procedure is defined similarly. We again use
vectors, now denoted $w=(w_1,\ldots,w_n)$, with coordinates taking
one of two different values. However, in order to distinguish
between the two procedures, we let $w_j\in\{\ast,\times\}$. As
before, we let $\hat{{{w}}}=(w_1, \ldots,w_{n-1})$.

Once again $Y^{w}_t$ is defined recursively from $Y^{\hat{w}}_t$ and
an $\alpha$-stable random measure $M_1$ with $\alpha\in(1,2]$; however,\vadjust{\goodbreak}
the control measure of $M_1$ is no longer Lebesgue measure as it was
in the case of $M_0$. Suppose that $(\Omega',\FF',\P')$ is the
probability space of ${Y}^{\hat{w}}_t$. Then, just as in
(\ref{defGamma}), $M_1(d\omega'\times dx)$ has control measure
$\P'\times\mbox{Lebesgue}$ and lives on some other probability space
$(\Omega,\FF,\P)$. If $w_n=(\ast)$, we let
%
\begin{equation}
\label{yp2} Y^{w}_t:= \int_{\Omega'\times\R}
\ell_{Y^{\hat{w}}}(t,x) \bigl(\omega'\bigr) M_1\bigl(d
\omega'\times dx\bigr),
\end{equation}
and if $w_n=(\times)$, we let
%
\begin{equation}
\label{ym2} Y^{w}_t:= \int_{\Omega'\times\R}
1_{[0,Y^{\hat{w}}_t(\omega')]}(x) M_1\bigl(d\omega'\times dx\bigr).
\end{equation}

We must\vspace*{1pt} show that the above recursions makes sense, that is, that the
integrals are well defined. In general, it is known that $H$-sssi
{S$\al$S} processes have $L^2(\R)$ local times almost surely. This
almost gets us to where we want to be; however, there are two
separate issues with which we must deal.

According to (\ref{defint}) we need that the integral kernels of
(\ref{yp}) and (\ref{ym}) are in $L^\alpha(\R)$ [which easily follows
if they are in $L^2(\R)$], but (\ref{yp}) and (\ref{ym}) are not in
general {S$\al$S} processes, and thus we need an extra argument to
show that
they have $L^2(\R)$ local times almost surely.

The second issue concerns (\ref{yp2}) and (\ref{ym2}) which are
{S$\al$S} processes, but are well defined only if the local times are in
$L^\alpha(\Om'\times\R)$. In other words, we will need the $\alpha$th
moment of the local times to be integrable. To solve these two
issues, we use the following result.
%
\begin{theorem} \label{pp}
\mbox{Suppose $Y_t=Y_t(\omega')$ is an $H'$-sssi process which satisfies:}

\begin{longlist}[(a)]
\item[(a)] $0< \E|Y_1| < \infty$.
\item[(b)] $Y_t$ has a local time satisfying $0< \E\int_{\R}\ell
_Y(t,x)^2\,dx < \infty$.
\item[(c)] $Y_1$ has a bounded continuous density.
\item[(d)] $\E[\sup_{t \in[0,1]} |Y_t|]<\infty$.
\end{longlist}
Then the processes
%
\begin{eqnarray}
Y^{(+)}_t &=& \int_{\R}
\ell_Y(t,x) M_0(dx),
\\
\label{minus}
Y^{(-)}_t &=& \int_{\R}
1_{[0,Y_t]}(x) M_0(dx),
\\
Y^{(\ast)}_t &=& \int_{\Om'\times\R}
\ell_Y(t,x) \bigl(\omega'\bigr) M_1\bigl(d
\omega'\times dx\bigr),
\\
\label{times}
Y^{(\times)}_t &=& \int_{\Om'\times\R}
1_{[0,Y_t(\omega')]}(x) M_1\bigl(d\omega'\times dx\bigr)
\end{eqnarray}
are\vspace*{1pt} well-defined $H$-sssi
processes satisfying \textup{(b)--(d)}, where $H = 1 - H' + H'/\alpha$ for
$Y^{(+)}$ and $Y^{(\ast)}$ and $H =H'/\alpha$ for $Y^{(-)}$ and
$Y^{(\times)}$. Moreover, all four processes have finite $\alpha$
moments which implies they also satisfy \textup{(a)}.\vadjust{\goodbreak}
\end{theorem}
\begin{remarks*}
(1)
For the proof, we need that $\ell_Y$ satisfies the occupation time
formula
\[
\int_{0}^{t} 1_A(Y_s)
\,ds = \int_{A} \ell_Y(t,x)\,dx
\]
for any Borel set $A$. This follows from definition (\ref{ltdef}).

(2) The processes $Y^{(+)}_t$ and $Y^{(-)}_t$ are not generally
stable. However, as mentioned above, when they have finite $\alpha$ moments,
one can use the stable central limit theorem and normalize partial
sums of independent copies of these processes to get stable
processes.
\end{remarks*}
\begin{pf*}{Proof of Theorem~\ref{pp}}
\textit{Well-defined $H$-sssi processes with finite $\alpha$th
moments.}\quad
To see that the $Y^\bullet_t$ are well defined and satisfy (a), we
have
%
\begin{equation}
\E\bigl[\bigl(Y^{(-)}_1\bigr)^\alpha\bigr]=\E\bigl[
\bigl(Y^{(\times)}_1\bigr)^\alpha\bigr]=\E'
\int_{\R} \bigl|1_{[0,Y_1]}(x)\bigr|^\alpha \,dx =
\E' |Y_1|,
\end{equation}
which is positive and finite
since $Y_t$ satisfies (a). Also,
%
\begin{equation}
\label{65} \E\bigl[\bigl(Y^{(+)}_1\bigr)^\alpha
\bigr] = \E\bigl[\bigl(Y^{(\ast)}_1\bigr)^\alpha\bigr] =
\E'\int_{\R} \ell_Y(1,x)^\alpha
\,dx.
\end{equation}
To see
that (\ref{65}) is finite and nonzero, note that $\E'\int_{\R}
\ell_Y(1,x)\,dx = 1$ by the occupation time formula and $\E'\int_{\R}
\ell_Y(1,x)^2 \,dx < \infty$ by (b), thus $\ell_Y\in L^\alpha(\Om'\times\R)$
for $\alpha\in(1,2]$.

To see that the $Y^\bullet$ are $H$-sssi, we refer the reader to
Theorem 3.1 in~\cite{CS} and Theorem 2.2 in
\cite{jung2010indicator}.\vspace*{8pt}

\textit{Property} (b).\quad
Next we use Theorem 21.9 of~\cite{geman1980occupation} which implies
condition (b) under the assumption that
%
\begin{equation}
\label{charep} \int_{\reals}\int_0^T
\int_0^T \E\bigl[e^{i\th(Y^\bullet_t-Y^\bullet_s)}\bigr]\,ds \,dt \,d
\th<\infty.
\end{equation}

Let us show (b) for $Y^{(\ast)}_t$. We have
%
\begin{equation}
\label{rm} \E\bigl[e^{i\th(Y^{(\ast)}_t-Y^{(\ast)}_s)}\bigr] = \exp
\biggl(-\th^\alpha
\E'\int_{\R}\bigl|\ell_Y(t,x)-
\ell_Y(s,x)\bigr|^\alpha \,dx \biggr).
\end{equation}
Using
$\ell_Y(t,x-B_s)-\ell_Y(s,x-B_s) \ed\ell_Y(t-s,x)$ and
%
\begin{equation}
\label{ltscale} \ell_Y\bigl(ct,c^{H'}x\bigr) \stackrel{d}
{=} c^{1-H'}\ell_Y(t,x)
\end{equation}
we see that
%
\begin{eqnarray}
\label{scale} \int_{\reals}\bigl|\ell_Y(t,x)-
\ell_Y(s,x)\bigr|^\alpha \,dx &\ed& \int_{\reals}
\ell_Y(t-s,x)^\alpha \,dx
\nonumber\\[-8pt]\\[-8pt]
&\ed& \int_{\reals}|t-s|^{\alpha(1-H'+H'/\alpha)}\ell_Y(1,u)^\alpha
\,du.\nonumber
\end{eqnarray}
Substituting $v=\theta\cdot
|t-s|^{1-H'+H'/\alpha} (\E'\int_{\reals}\ell_Y(1,u)^\alpha
\,du )^{1/\alpha}$
we get that (\ref{charep}) equals
%
\begin{equation}\quad
\biggl(\E'\int_{\reals}\ell_Y(1,u)^\alpha
\,du \biggr)^{-1/\alpha}\int_{\reals}e^{-v^\alpha}\int
_0^T\int_0^T
|t-s|^{-1+H'-H'/\alpha} \,ds \,dt \,dv,
\end{equation}
which is finite since
$\E'\int_{\reals}\ell_Y(1,u)^\alpha \,du>0$ by the occupation time
formula.

To show (b) for $Y^{(+)}_t$, write
%
\begin{eqnarray}
\label{rm2}\qquad \E\bigl[e^{i\th(Y^{(+)}_t-Y^{(+)}_s)}\bigr] & = & \int_{\Om'}
\int_{\Om} \exp\biggl({i\th\int_{\R}
\bigl(\ell_Y(t,x)-\ell_Y(s,x)\bigr) M_0(dx)}
\biggr) \,d\omega \,d\omega'
\nonumber\\[-8pt]\\[-8pt]
& = &\E' \exp\biggl({-\th^\alpha\int_{\R}\bigl|
\ell_Y(t,x)-\ell_Y(s,x)\bigr|^\alpha \,dx} \biggr).\nonumber
\end{eqnarray}
Using (\ref{scale}) and (\ref{rm2}), we have that, in this case,
(\ref{charep}) is
\[
\int_0^T\int_0^T
\E' \biggl[\int_{\reals} \exp\biggl({-
\th^\alpha\int_{\reals}|t-s|^{\alpha(1-H'+H'/\alpha)}
\ell_Y(1,u)^\alpha \,du} \biggr)\,d\theta\biggr] \,ds \,dt.
\]

Substituting $v=\theta\cdot
|t-s|^{1-H'+H'/\alpha} (\int_{\reals}\ell_Y(1,u)^\alpha
\,du )^{1/\alpha}$ and
integrating we obtain, for some constant $c>0$,
%
\begin{equation}
\label{rer} c \E' \biggl[ \biggl(\int_{\reals}
\ell_Y(1,u)^\alpha \,du \biggr)^{-1/\alpha} \biggr]\int
_0^T\int_0^T
\frac
{ds\,dt}{|t-s|^{1-H'+H'/\alpha}}.
\end{equation}
To show that this is finite we need only show that $\E'
[(\int_{\reals}\ell_Y(1,u)^\alpha \,du)^{-1/\alpha} ]<
\infty$. We have
$\ell_Y(1,x)(\omega') = 0$ for
%
\begin{equation}
|x|> A\bigl(\omega'\bigr):= \sup_{t\in[0,1]} \bigl|Y_t
\bigl(\omega'\bigr)\bigr|,
\end{equation}
so by Hold\"{e}r's inequality,
%
\begin{eqnarray}
\label{71}\qquad \int_{-\infty}^{\infty} \ell_Y(1,x)
\bigl(\omega'\bigr) \,dx &=& \int_{-A(\omega')}^{A(\omega')}
\ell_Y(1,x) \bigl(\omega'\bigr) \,dx
\nonumber\\[-8pt]\\[-8pt]
&\leq&\bigl(2A\bigl(\omega'\bigr)\bigr)^{(\alpha-1)/\alpha}
\biggl(\int_{-A(\omega
')}^{A(\omega')} \bigl(\ell_Y(1,x)
\bigl(\omega'\bigr) \bigr)^\alpha \,dx \biggr)^{1/\alpha}.
\nonumber
\end{eqnarray}
By the occupation
time formula, the left-hand side of (\ref{71}) equals $1$ a.s. so that
%
\begin{equation}
\E' \biggl(\int_\R
\ell_Y(1,x)^\alpha \,dx \biggr)^{-1/\alpha} \leq
\E' (2A)^{(\alpha-1)/\alpha}.
\end{equation}
Property (d) of $Y_t$ completes the proof of
(b) for $Y^{(+)}$.

Moving on to $Y^{(\times)}_t$, we have
%
\begin{eqnarray}
\E\bigl[e^{i\th(Y^{(\times)}_t-Y^{(\times)}_s)}\bigr] &=& \exp\biggl
(-\th^\alpha
\E'\int_{\R}1_{[Y_s,Y_t]} \,dx \biggr)
\nonumber\\[-8pt]\\[-8pt]
&=&\exp\bigl(-\th^\alpha\E'|Y_{t-s}| \bigr)=\exp
\bigl(-\th^\alpha|t-s|^{H'}\E'|Y_{1}|
\bigr).\nonumber
\end{eqnarray}
Thus (\ref{charep}) reduces to
%
\begin{eqnarray}
&&\int_{\R} \int_{0}^{T}
\int_{0}^{T} \exp\bigl(-\th^\alpha
|t-s|^{H'}\E'|Y_{1}| \bigr) \,dt\,ds\,d\th\nonumber\\[-8pt]\\[-8pt]
&&\qquad= C\int
_{0}^{T}\int_{0}^{T}
|t-s|^{-H'/\alpha} \,dt\,ds < \infty,\nonumber
\end{eqnarray}
where $C = \int_{\reals} \exp(-u^\alpha\E'|Y_{1}|) \,du$.

Finally, let us consider $Y^{(-)}$. We may mimic steps (\ref{rm})
through (\ref{rer}) in order to reduce (\ref{charep}) to showing
%
\begin{equation}
\E' \biggl[ \biggl(\int_{\reals}1_{[0,Y_1]}(x)
\,dx \biggr)^{-1/{\alpha
}} \biggr]= \E'\bigl[|Y_1|^{-1/\alpha}
\bigr] < \infty.
\end{equation}
But this follows from assumption (c)
on $Y_t$, since we may simply integrate $|x|^{-1/\alpha}$ against the
bounded continuous density of $Y_1$ which will give a finite value.
This establishes (b) for $Y^{(-)}$.\vspace*{8pt}

\textit{Property} (c).\quad
In the course of showing property (b) for $Y^\bullet_t$, we showed
that in all cases $Y_t^\bullet$ possesses a nonnegative and
integrable characteristic function, and thus (c) follows from
Theorem 3.3.5 in~\cite{durrett2010probability}.\vspace*{8pt}

\textit{Property} (d).\quad
Consider first $\alpha=2$. Property (d) is known for $Y_t^{(\ast)}$ and
$Y_t^{(\times)}$ since they are sssi Gaussian processes, that is,
fractional Brownian motions.

For $Y_t^{(+)}$, let $\tilde{B}_t$ be a two-sided Brownian motion.
We use Proposition 2.2 in~\cite{khoshnevisan1998law} which is
essentially a corollary of Slepian's lemma. It implies that for each
fixed $\omega'$,
%
\begin{equation}\quad
\P\biggl(\sup_{t\in[0,1]} \int_{\R}
1_{[0,Y_t(\omega')]}(s) \,d\tilde{B}_s > y \biggr) \leq2\P\biggl(\int
_{\R} 1_{[0,Y_1(\omega')]}(s) \,d\tilde{B}_s > y
\biggr).
\end{equation}
Integrating over $\Om'$, property (d) for $Y^{(+)}_t$ follows from
property (d) for $Y_t$.

For $Y_t^{(-)}$, let $Y^*:=\sup_{t \in[0,1]} Y_t$, and
$Y_*:=\inf_{t \in[0,1]} Y_t$. We have
%
\begin{eqnarray}
\E\Bigl[\sup_{t \in[0,1]} \bigl|Y_t^{(-)}\bigr|\Bigr] &\le& \E
\Bigl[\sup_{t \in[0,1]} Y_t^{(-)}+\sup_{t \in[0,1]}
\bigl(-Y_t^{(-)}\bigr)\Bigr]
\nonumber
\\
&\le& 2\E'\biggl[\sup_{t \in[0,1]} \int_{\R}
1_{[0,Y_t]}(s) \,d\tilde{B}_s\biggr]
\nonumber\\[-8pt]\\[-8pt]
&\le& 2\E'\biggl[\sup_{T \in[Y_*,Y^*]} \int_{\R}
1_{[0,T]}(s) \,d\tilde{B}_s\biggr]
\nonumber
\\
&\le& 8\E'\bigl(Y^*\bigr).\nonumber
\end{eqnarray}
The last inequality follows since the integral
in the second to last line is just a two-sided Brownian motion at
time $T$ and $\E'(Y^*)=\E'(-Y_*)<\infty$. We thus get property (d) for
$Y^{(-)}_t$ since property (d) holds for $Y_t$.

Let us now suppose that $1<\alpha<2$. Theorem 10.5.1 of
\cite{samorodnitsky1994stable} states that if
%
\begin{equation}
Y_t = \int_{E} f_t(x)
M(dx)
\end{equation}
for some family of $L^\alpha(E,m)$ functions $\{f_t(x)\}_{t\ge0}$,
where $m$ is the control measure of $M$, then there is a constant
$C$ such that
%
\begin{equation}
\P\Bigl(\sup_{t\in[0,1]} |Y_t| > y\Bigr) \leq
\frac{C}{y^\alpha} \int_{E} \sup_{t\in[0,1]}
\bigl|f_t(x)\bigr|^\alpha m(dx)
\end{equation}
for any $y >0$.

We can therefore obtain (d) for $Y_t^\bullet$ by showing that
%
\begin{equation}
\label{ff1} \E'\int_{\R} \sup_{t\in[0,1]}
\bigl(\ell_Y(t,x)^\alpha\bigr) \,dx = \E'\int
_{\R} \ell_Y(1,x)^\alpha \,dx
\end{equation}
and
%
\begin{equation}
\label{ff2} \E'\int_{\R} \sup_{t\in[0,1]}
\bigl(1_{[0,Y_t]}(x)^\alpha\bigr) \,dx = 2\E' \Bigl(
\sup_{t\in[0,1]} Y_t \Bigr)
\end{equation}
are both finite. As seen in (\ref{65}) and the argument thereafter,
(\ref{ff1}) is finite since $Y_t$ satisfies (b). Also (\ref{ff2}) is
finite since $Y_t$ satisfies (d).
\end{pf*}

For fixed $\alpha\in(1,2]$, define
%
\begin{equation}
\label{defphi} \phi_+(x):= 1-x+x/\alpha\quad\mbox{and}\quad \phi_-(x):=x/\alpha.
\end{equation}
Applying Theorem
\ref{pp} recursively, we have the following corollary:
%
\begin{corollary} \label{t3}
If $Y^\varnothing_t$ is an $H'$-sssi process satisfying \textup{(a)--(d)} of
Theorem~\ref{pp}, then $Y_t^{(v_1,\ldots,v_{n})}$ and
$Y_t^{(w_1,\ldots,w_{n})}$ [as defined in (\ref{yp})--
(\ref{ym2})] are $H$-sssi processes with
\[
H=\phi_{v_n}\circ\cdots\circ\phi_{v_1}\bigl(H'
\bigr).
\]
Moreover, $Y_t^{(w_1,\ldots,w_{n})}$ is an {S$\al$S} process.
\end{corollary}

\section{Brownian motion extracted from fBm, $H<1/2$}\label{sectime-changes}
Suppose $\alpha=2$. Then the family of stochastic integrals,
$(Y^{(\times)}_{t})_{t\ge0}$, is an $H'$-sssi Gaussian process,
thus it is precisely fBm with Hurst exponent $H' < 1/2$. In this
section, we show that Brownian motion can be extracted from
$Y^{(\times)}_t$ by time-changing its integral kernels. In order to
motivate our time-changed kernels, we first show that Brownian
motion can also be extracted from a stable process at random time,
$Y^{(-)}_t$, using a time-change.

To keep things simple, we assume in this section that the random
time process $Y_t$ is itself an fBm. Thus it is a.s. continuous and
satisfies the property that for each $s>0$,
%
\begin{equation}
\label{finitetau} \tau_s = \inf_{t \geq0}
\{t\dvtx Y_t=s\}<\infty\qquad\mbox{a.s.}
\end{equation}

Heuristically, time-changing the kernel of $Y^{\bullet}_t$ undoes
the subordination of $Y^{\bullet}_t$ to the process $Y_t$, leaving
us with a process $(M(A_t))_{t\ge0}$. We then observe that
$A_s\subset A_t$ for $s<t$, and that $m(A_t)$ is linearly increasing
(here $m$ is the control measure). One need only check that such a
procedure gives us what we want, by looking at the finite-dimensional distributions. Since our interest is in the case
$\alpha=2$, we have that $M_0$, $M_1$ are Gaussian random measures on
$\reals$ and $\Om' \times\reals$, respectively, and we in fact need
only check covariances.

Let us start by presenting the time-change of $Y^{(-)}_t$.
%
\begin{prop} \label{lastprop}
Let the random time process $Y_t$ be a fractional Brownian motion.
If $Y^{(-)}_t$ is defined as in (\ref{minus}) with $\alpha=2$, then
$Y^{(-)}_{\tau_t}$ is a Brownian motion.
\end{prop}
\begin{pf}
We have
%
\begin{eqnarray}
Y^{(-)}_{\tau_{t}} = \int_{\reals}
1_{[0,
Y_{\tau_{t}}]}(x) M_0(dx) = \int_{\reals}
1_{[0,t]}(s)\,d\tilde{B}_s = \tilde B_{t},
\end{eqnarray}
where $\tilde{B}_t$ is a two-sided Brownian
motion. For the covariances, if $s<t$, we have
%
\begin{eqnarray}
\E\bigl(Y^{(-)}_{\tau_{s}}Y^{(-)}_{\tau_{t}}
\bigr)
&=& \E\biggl(\int_{\reals} 1_{[0,s]}(r)\,d
\tilde{B}_r\cdot\int_{\reals} 1_{[0,t]}(r)\,d
\tilde{B}_r \biggr)
\nonumber\\[-8pt]\\[-8pt]
&=& \int_{\reals} \bigl(1_{[0,s]}(r)
\bigr)^2 \,dr = s.\nonumber
\end{eqnarray}
\upqed\end{pf}

In the case of
\[
Y^{(\times)}_t = \int_{\Om'\times\R}
1_{[0,Y_t(\omega')]}(x) M_1\bigl(d\omega'\times dx\bigr),
\]
we cannot look at ``$Y^{(\times)}_{\tau_t}$'' since $\tau_t$ lives
on the same probability space as~$M_1$. We address this issue by\vadjust{\goodbreak}
instead time-changing the kernel $1_{[0,Y_t]}$. Let us define
%
\begin{equation}
\label{lasteq} Y^{(\times)_\tau}_t := \int_{\Om'\times\R}
1_{[0,Y_{\tau_t}(\omega')]}(x) M_1\bigl(d\omega'\times dx\bigr).
\end{equation}
A good way to think about the above integral is in terms of
a central limit theorem similar to (\ref{resultCLT}):
%
\begin{equation}\label{resultCLT2}
n^{-1/2}\sum_{i=1}^n
\Delta_H\bigl(\tau_t^{(i)}\bigr)^{(i)}
\stackrel{\mathrm{f.d.d.}}{\lar} \int_{\Om'\times\R} 1_{[0,Y_{\tau_t}(\omega')]}(x) M_1\bigl(d
\omega'\times dx\bigr).
\end{equation}
Here, $\tau^{(i)}$ is measurable with respect to the $\sigma$-field of
$\Delta_H^{(i)}$. By Proposition~\ref{lastprop}, the
$\Delta_H(\tau_t^{(i)})^{(i)}$ are independent Brownian motions.
The next proposition shows that the right-hand side is also a Brownian
motion thus proving~(\ref{resultCLT2}).

\begin{prop}
Let the random time process $Y_t$ be a fractional Brownian motion.
If $Y^{(\times)}_t$ is defined as in (\ref{times}) with $\alpha=2$, then
$Y^{(\times)_\tau}_{t}$ is a Brownian motion.
\end{prop}
\begin{pf}
We have
%
\begin{eqnarray}
Y^{(\times)_\tau}_{t} &=& \int_{\Om' \times\reals}
1_{[0,
Y_{\tau_t}]}(x) M_1\bigl(d\omega' \times dx\bigr)
\nonumber\\[-8pt]\\[-8pt]
&=& \int_{\Om\times\reals} 1_{[0,t]}(x) M_1\bigl(d
\omega' \times dx\bigr) = M_1\bigl(\Om'
\times[0,t]\bigr),\nonumber
\end{eqnarray}
which is a Gaussian random variable with
variance $\P'\times\mathrm{Leb}(\Om' \times[0,t]) = t$. For the
covariances we analyze second moments. If $s<t$, we have
%
\begin{eqnarray}
\E\bigl(Y^{(\times)_\tau}_{s}+Y^{(\times)_\tau
}_{t}
\bigr)^2
&=& \E\biggl(\int_{\Om'\times\reals} (2
\cdot1_{[0,s]}+1_{[s,t]} ) M_1\bigl(d
\omega' \times dx\bigr) \biggr)^2
\nonumber\\
&=& \int_{\Om'\times\reals} (2\cdot1_{[0,s]}+1_{[s,t]}
)^2 \P'\times\mathrm{Leb}\bigl(d\omega'
\times dx\bigr)
\nonumber\\[-8pt]\\[-8pt]
&=& \int_{\reals} (4\cdot1_{[0,s]}+1_{[s,t]}
) \,dx
\nonumber\\
&=& 3s+t\nonumber
\end{eqnarray}
as required.
\end{pf}

\section*{Acknowledgments}

We thank Clement Dombry, Harry Kesten and Gennady Samorodnitsky for
helpful email correspondence. We also thank anonymous referees for
careful readings, nice suggestions and corrections.

Part of this work was done when P. Jung was visiting Pohang
University of Science and Technology and also IPAM at UCLA. He
thanks both institutions for their hospitality. G. Markowsky thanks
Sogang University for hospitality during which some of this work was
done.\vadjust{\goodbreak}



\printaddresses


\begin{thebibliography}{31}

\bibitem{berman1974}
\begin{barticle}[mr]
\bauthor{\bsnm{Berman},~\bfnm{Simeon~M.}\binits{S.~M.}}
(\byear{1974}).
\btitle{Local nondeterminism and local times of {G}aussian processes}.
\bjournal{Indiana Univ. Math. J.}
\bvolume{23}
\bpages{69--94}.
\bid{issn={0022-2518}, mr={0317397}}
\bptok{imsref}%
\end{barticle}
\endbibitem


\bibitem{Borodin}
\begin{barticle}[mr]
\bauthor{\bsnm{Borodin},~\bfnm{A.~N.}\binits{A.~N.}}
(\byear{1979}).
\btitle{Limit theorems for sums of independent random variables defined on a
  transient random walk}.
\bjournal{Zap. Nauchn. Sem. Leningrad. Otdel. Mat. Inst. Steklov. (LOMI)}
\bvolume{85}
\bpages{17--29, 237, 244}.
\bnote{Investigations in the theory of probability distributions, IV}.
\bid{issn={0207-6772}, mr={0535455}}
\bptok{imsref}%
\end{barticle}
\endbibitem

\bibitem{boylan1964local}
\begin{barticle}[mr]
\bauthor{\bsnm{Boylan},~\bfnm{Edward~S.}\binits{E.~S.}}
(\byear{1964}).
\btitle{Local times for a class of {M}arkoff processes}.
\bjournal{Illinois J. Math.}
\bvolume{8}
\bpages{19--39}.
\bid{issn={0019-2082}, mr={0158434}}
\bptok{imsref}%
\end{barticle}
\endbibitem

\bibitem{burdzy1992some}
\begin{bincollection}[mr]
\bauthor{\bsnm{Burdzy},~\bfnm{Krzysztof}\binits{K.}}
(\byear{1993}).
\btitle{Some path properties of iterated {B}rownian motion}.
In \bbooktitle{Seminar on {S}tochastic {P}rocesses, 1992 ({S}eattle, {WA},
  1992)}.
\bseries{Progress in Probability}
\bvolume{33}
\bpages{67--87}.
\bpublisher{Birkh\"auser}, \blocation{Boston, MA}.
\bid{mr={1278077}}
\bptnote{check year}%
\bptok{imsref}%
\end{bincollection}
\endbibitem

\bibitem{CD}
\begin{barticle}[mr]
\bauthor{\bsnm{Cohen},~\bfnm{Serge}\binits{S.}} \AND
  \bauthor{\bsnm{Dombry},~\bfnm{Cl{\'e}ment}\binits{C.}}
(\byear{2009}).
\btitle{Convergence of dependent walks in a random scenery to f{B}m-local time
  fractional stable motions}.
\bjournal{J. Math. Kyoto Univ.}
\bvolume{49}
\bpages{267--286}.
\bid{issn={0023-608X}, mr={2571841}}
\bptok{imsref}%
\end{barticle}
\endbibitem

\bibitem{CS}
\begin{barticle}[mr]
\bauthor{\bsnm{Cohen},~\bfnm{Serge}\binits{S.}} \AND
  \bauthor{\bsnm{Samorodnitsky},~\bfnm{Gennady}\binits{G.}}
(\byear{2006}).
\btitle{Random rewards, fractional {B}rownian local times and stable
  self-similar processes}.
\bjournal{Ann. Appl. Probab.}
\bvolume{16}
\bpages{1432--1461}.
\bid{doi={10.1214/105051606000000277}, issn={1050-5164}, mr={2260069}}
\bptok{imsref}%
\end{barticle}
\endbibitem

\bibitem{deheuvels1992functional}
\begin{bincollection}[mr]
\bauthor{\bsnm{Deheuvels},~\bfnm{Paul}\binits{P.}} \AND
  \bauthor{\bsnm{Mason},~\bfnm{David~M.}\binits{D.~M.}}
(\byear{1992}).
\btitle{A functional {LIL} approach to pointwise {B}ahadur--{K}iefer theorems}.
In \bbooktitle{Probability in {B}anach Spaces, 8 ({B}runswick, {ME}, 1991)}.
\bseries{Progress in Probability}
\bvolume{30}
\bpages{255--266}.
\bpublisher{Birkh\"auser}, \blocation{Boston, MA}.
\bid{mr={1227623}}
\bptok{imsref}%
\end{bincollection}
\endbibitem

\bibitem{den2006random}
\begin{bincollection}[mr]
\bauthor{\bparticle{den} \bsnm{Hollander},~\bfnm{Frank}\binits{F.}} \AND
  \bauthor{\bsnm{Steif},~\bfnm{Jeffrey~E.}\binits{J.~E.}}
(\byear{2006}).
\btitle{Random walk in random scenery: A~survey of some recent results}.
In \bbooktitle{Dynamics \& Stochastics}.
\bseries{Institute of Mathematical Statistics Lecture Notes---Monograph Series}
\bvolume{48}
\bpages{53--65}.
\bpublisher{Inst. Math. Statist.}, \blocation{Beachwood, OH}.
\bid{doi={10.1214/lnms/1196285808}, mr={2306188}}
\bptok{imsref}%
\end{bincollection}
\endbibitem

\bibitem{dombrySWN}
\begin{bmisc}[auto:STB|2012/08/23|07:51:16]
\bauthor{\bsnm{Dombry},~\bfnm{C.}\binits{C.}}
(\byear{2011}).
\bhowpublished{Discrete approximation of stable white noise: Applications to
  spatial linear filtering. Preprint.}
\bptok{imsref}%
\end{bmisc}
\endbibitem

\bibitem{DG}
\begin{barticle}[mr]
\bauthor{\bsnm{Dombry},~\bfnm{C.}\binits{C.}} \AND
  \bauthor{\bsnm{Guillotin-Plantard},~\bfnm{N.}\binits{N.}}
(\byear{2009}).
\btitle{Discrete approximation of a stable self-similar stationary increments
  process}.
\bjournal{Bernoulli}
\bvolume{15}
\bpages{195--222}.
\bid{doi={10.3150/08-BEJ147}, issn={1350-7265}, mr={2546804}}
\bptnote{check year}%
\bptok{imsref}%
\end{barticle}
\endbibitem

\bibitem{dombry14functional}
\begin{barticle}[mr]
\bauthor{\bsnm{Dombry},~\bfnm{C.}\binits{C.}} \AND
  \bauthor{\bsnm{Guillotin-Plantard},~\bfnm{N.}\binits{N.}}
(\byear{2009}).
\btitle{A functional approach for random walks in random sceneries}.
\bjournal{Electron. J. Probab.}
\bvolume{14}
\bpages{1495--1512}.
\bid{doi={10.1214/EJP.v14-659}, issn={1083-6489}, mr={2519528}}
\bptok{imsref}%
\end{barticle}
\endbibitem

\bibitem{durrett2010probability}
\begin{bbook}[mr]
\bauthor{\bsnm{Durrett},~\bfnm{Rick}\binits{R.}}
(\byear{2010}).
\btitle{Probability: Theory and Examples},
\bedition{4th} ed.
\bpublisher{Cambridge Univ. Press}, \blocation{Cambridge}.
\bid{mr={2722836}}
\bptok{imsref}%
\end{bbook}
\endbibitem

\bibitem{funaki1979probabilistic}
\begin{barticle}[mr]
\bauthor{\bsnm{Funaki},~\bfnm{Tadahisa}\binits{T.}}
(\byear{1979}).
\btitle{Probabilistic construction of the solution of some higher order
  parabolic differential equation}.
\bjournal{Proc. Japan Acad. Ser. A Math. Sci.}
\bvolume{55}
\bpages{176--179}.
\bid{issn={0386-2194}, mr={0533542}}
\bptok{imsref}%
\end{barticle}
\endbibitem

\bibitem{geman1980occupation}
\begin{barticle}[mr]
\bauthor{\bsnm{Geman},~\bfnm{Donald}\binits{D.}} \AND
  \bauthor{\bsnm{Horowitz},~\bfnm{Joseph}\binits{J.}}
(\byear{1980}).
\btitle{Occupation densities}.
\bjournal{Ann. Probab.}
\bvolume{8}
\bpages{1--67}.
\bid{issn={0091-1798}, mr={0556414}}
\bptok{imsref}%
\end{barticle}
\endbibitem

\bibitem{guillotinlimit}
\begin{barticle}[mr]
\bauthor{\bsnm{Guillotin-Plantard},~\bfnm{Nadine}\binits{N.}} \AND
  \bauthor{\bsnm{Prieur},~\bfnm{Cl{\'e}mentine}\binits{C.}}
(\byear{2010}).
\btitle{Limit theorem for random walk in weakly dependent random scenery}.
\bjournal{Ann. Inst. Henri Poincar\'e Probab. Stat.}
\bvolume{46}
\bpages{1178--1194}.
\bid{doi={10.1214/09-AIHP353}, issn={0246-0203}, mr={2744890}}
\bptok{imsref}%
\end{barticle}
\endbibitem

\bibitem{jung2010indicator}
\begin{barticle}[mr]
\bauthor{\bsnm{Jung},~\bfnm{Paul}\binits{P.}}
(\byear{2011}).
\btitle{Indicator fractional stable motions}.
\bjournal{Electron. Commun. Probab.}
\bvolume{16}
\bpages{165--173}.
\bid{doi={10.1214/ECP.v16-1611}, issn={1083-589X}, mr={2783337}}
\bptok{imsref}%
\end{barticle}
\endbibitem

\bibitem{KS}
\begin{barticle}[mr]
\bauthor{\bsnm{Kesten},~\bfnm{H.}\binits{H.}} \AND
  \bauthor{\bsnm{Spitzer},~\bfnm{F.}\binits{F.}}
(\byear{1979}).
\btitle{A limit theorem related to a new class of self-similar processes}.
\bjournal{Probab. Theory Related Fields}
\bvolume{50}
\bpages{5--25}.
\bptok{imsref}%
\end{barticle}
\endbibitem

\bibitem{khoshnevisan1998law}
\begin{barticle}[mr]
\bauthor{\bsnm{Khoshnevisan},~\bfnm{Davar}\binits{D.}} \AND
  \bauthor{\bsnm{Lewis},~\bfnm{Thomas~M.}\binits{T.~M.}}
(\byear{1998}).
\btitle{A law of the iterated logarithm for stable processes in random
  scenery}.
\bjournal{Stochastic Process. Appl.}
\bvolume{74}
\bpages{89--121}.
\bid{doi={10.1016/S0304-4149(97)00105-1}, issn={0304-4149}, mr={1624017}}
\bptok{imsref}%
\end{barticle}
\endbibitem

\bibitem{khoshnevisan1996iterated}
\begin{bincollection}[mr]
\bauthor{\bsnm{Khoshnevisan},~\bfnm{Davar}\binits{D.}} \AND
  \bauthor{\bsnm{Lewis},~\bfnm{Thomas~M.}\binits{T.~M.}}
(\byear{1999}).
\btitle{Iterated {B}rownian motion and its intrinsic skeletal structure}.
In \bbooktitle{Seminar on {S}tochastic {A}nalysis, {R}andom {F}ields and
  {A}pplications ({A}scona, 1996)}.
\bseries{Progress in Probability}
\bvolume{45}
\bpages{201--210}.
\bpublisher{Birkh\"auser}, \blocation{Basel}.
\bid{mr={1712242}}
\bptnote{check year}%
\bptok{imsref}%
\end{bincollection}
\endbibitem

\bibitem{khoshnevisan1999stochastic}
\begin{barticle}[mr]
\bauthor{\bsnm{Khoshnevisan},~\bfnm{Davar}\binits{D.}} \AND
  \bauthor{\bsnm{Lewis},~\bfnm{Thomas~M.}\binits{T.~M.}}
(\byear{1999}).
\btitle{Stochastic calculus for {B}rownian motion on a {B}rownian fracture}.
\bjournal{Ann. Appl. Probab.}
\bvolume{9}
\bpages{629--667}.
\bid{doi={10.1214/aoap/1029962807}, issn={1050-5164}, mr={1722276}}
\bptok{imsref}%
\end{barticle}
\endbibitem

\bibitem{legall1991range}
\begin{barticle}[mr]
\bauthor{\bsnm{Le~Gall},~\bfnm{Jean-Fran{\c{c}}ois}\binits{J.-F.}} \AND
  \bauthor{\bsnm{Rosen},~\bfnm{Jay}\binits{J.}}
(\byear{1991}).
\btitle{The range of stable random walks}.
\bjournal{Ann. Probab.}
\bvolume{19}
\bpages{650--705}.
\bid{issn={0091-1798}, mr={1106281}}
\bptok{imsref}%
\end{barticle}
\endbibitem

\bibitem{ledoux1991probability}
\begin{bbook}[mr]
\bauthor{\bsnm{Ledoux},~\bfnm{Michel}\binits{M.}} \AND
  \bauthor{\bsnm{Talagrand},~\bfnm{Michel}\binits{M.}}
(\byear{1991}).
\btitle{Probability in {B}anach Spaces: Isoperimetry and Processes}.
\bseries{Ergebnisse der Mathematik und Ihrer Grenzgebiete (3) [Results in
  Mathematics and Related Areas (3)]}
\bvolume{23}.
\bpublisher{Springer}, \blocation{Berlin}.
\bid{mr={1102015}}
\bptok{imsref}%
\end{bbook}
\endbibitem

\bibitem{nane2006laws}
\begin{barticle}[mr]
\bauthor{\bsnm{Nane},~\bfnm{Erkan}\binits{E.}}
(\byear{2006}).
\btitle{Laws of the iterated logarithm for {$\alpha $}-time {B}rownian motion}.
\bjournal{Electron. J. Probab.}
\bvolume{11}
\bpages{434--459 (electronic)}.
\bid{doi={10.1214/EJP.v11-327}, issn={1083-6489}, mr={2223043}}
\bptok{imsref}%
\end{barticle}
\endbibitem

\bibitem{nane2011local}
\begin{barticle}[mr]
\bauthor{\bsnm{Nane},~\bfnm{Erkan}\binits{E.}},
  \bauthor{\bsnm{Wu},~\bfnm{Dongsheng}\binits{D.}} \AND
  \bauthor{\bsnm{Xiao},~\bfnm{Yimin}\binits{Y.}}
(\byear{2012}).
\btitle{{$\alpha $}-time fractional {B}rownian motion: {PDE} connections and
  local times}.
\bjournal{ESAIM Probab. Stat.}
\bvolume{16}
\bpages{1--24}.
\bid{doi={10.1051/ps/2011103}, issn={1292-8100}, mr={2900521}}
\bptnote{check year}%
\bptok{imsref}%
\end{barticle}
\endbibitem

\bibitem{samorodnitsky2005null}
\begin{barticle}[mr]
\bauthor{\bsnm{Samorodnitsky},~\bfnm{Gennady}\binits{G.}}
(\byear{2005}).
\btitle{Null flows, positive flows and the structure of stationary symmetric
  stable processes}.
\bjournal{Ann. Probab.}
\bvolume{33}
\bpages{1782--1803}.
\bid{doi={10.1214/009117905000000305}, issn={0091-1798}, mr={2165579}}
\bptok{imsref}%
\end{barticle}
\endbibitem

\bibitem{samorodnitsky1994stable}
\begin{bbook}[mr]
\bauthor{\bsnm{Samorodnitsky},~\bfnm{Gennady}\binits{G.}} \AND
  \bauthor{\bsnm{Taqqu},~\bfnm{Murad~S.}\binits{M.~S.}}
(\byear{1994}).
\btitle{Stable Non-{G}aussian Random Processes:
Stochastic Models with Infinite Variance}.
\bpublisher{Chapman \& Hall}, \blocation{New York}.
\bid{mr={1280932}}
\bptok{imsref}%
\end{bbook}
\endbibitem

\bibitem{turban2004iterated}
\begin{barticle}[auto:STB|2012/08/23|07:51:16]
\bauthor{\bsnm{Turban},~\bfnm{L.}\binits{L.}}
(\byear{2004}).
\btitle{Iterated random walk}.
\bjournal{Europhysics Letters}
\bvolume{65}
\bpages{627}.
\bptok{imsref}%
\end{barticle}
\endbibitem

\bibitem{Wang03}
\begin{barticle}[mr]
\bauthor{\bsnm{Wang},~\bfnm{Wensheng}\binits{W.}}
(\byear{2003}).
\btitle{Weak convergence to fractional {B}rownian motion in {B}rownian
  scenery}.
\bjournal{Probab. Theory Related Fields}
\bvolume{126}
\bpages{203--220}.
\bid{doi={10.1007/s00440-002-0249-8}, issn={0178-8051}, mr={1990054}}
\bptok{imsref}%
\end{barticle}
\endbibitem


\end{thebibliography}
\end{document}